\documentclass[a4paper,11pt]{article}

\usepackage{epsfig,amsmath,amsfonts,amssymb,mathtools,amsthm,graphicx,xypic}
\usepackage{mathrsfs,enumerate,hyperref}%,eucal
\newcommand{\dst}{\displaystyle}

\newcommand{\sst}{\scriptstyle}
\newcommand{\ens}{\enspace}

\def\al{\alpha}
\def\be{\beta}
\def\ga{\gamma}
\def\Ga{{\Gamma}}
\def\de{\delta}

\def\De{\Delta}
\def\eps{{\varepsilon}}

\def\La{\Lambda}
\def\om{\omega}

\def\Om{\Omega}

\def\sig{{\sigma}}
\def\Sig{{\Sigma}}

\def\th{{\theta}}

\newcommand{\ph}{\varphi}

\def\ze{{\zeta}}
\newcommand{\dze}{\overset{\raisebox{-.23ex}{$\scriptscriptstyle\bullet$}}{\ze}}
\newcommand{\dzeph}{{\overset{\raisebox{-.23ex}{$\scriptscriptstyle\bullet$}}{\ze}}\vphantom{\ze'}}
\newcommand{\dxi}{{\overset{\raisebox{-.23ex}{$\scriptscriptstyle\bullet$}}{\xi}}\vphantom{\xi}}

\newcommand{\demi}{\frac{1}{2}}
\newcommand{\dem}{\tfrac{1}{2}}

\newcommand{\un}[1]{{\underline{#1}}}
\newcommand{\ov}{\overline}

\newcommand{\dist}{\operatorname{dist}}
\newcommand{\spt}{\operatorname{spt}}
\newcommand{\leng}{\operatorname{length}}

\newcommand{\Det}{\operatorname{det}}

\newcommand{\SUM}{\operatorname{Sum}}

\newcommand{\ID}{\mathop{\hbox{{\rm Id}}}\nolimits}

\newcommand{\I}{{\mathrm i}}
\newcommand{\dd}{{\mathrm d}}

\newcommand{\ee}{\mathrm e}

\newcommand{\pa}{\partial}

\newcommand{\ii}{^{-1}}
\newcommand{\ic}{^{\circ(-1)}}
\newcommand{\ti}{\tilde}

\newcommand{\bB}{\boldsymbol{B}}
\newcommand{\bbe}{\boldsymbol{e}}
\newcommand{\bj}{\boldsymbol{j}}
\newcommand{\bk}{\boldsymbol{k}}
\newcommand{\bn}{\boldsymbol{n}}

\newcommand{\RE}{\mathop{\Re e}\nolimits}

\def\ie{{\it i.e.}\ }

\def\eg{{\it e.g.}\ }
\def\resp{{resp.}\ }
\def\wrt{{with respect to}}
\def\lhs{{left-hand side}}
\def\rhs{{right-hand side}}

\def\dst{\displaystyle}

\def\sst{\scriptstyle}

\newcommand{\C}{\mathbb{C}}

\newcommand{\D}{\mathbb{D}}

\newcommand{\N}{\mathbb{N}}

\newcommand{\R}{\mathbb{R}}

\newcommand{\Z}{\mathbb{Z}}

\def\cA{\mathcal{A}}
\def\cB{\mathcal{B}}

\def\cF{\mathcal{F}}

\renewcommand{\cH}{\mathcal{H}}

\def\cK{\mathcal{K}}
\def\cL{\mathcal{L}}

\def\cN{\mathcal{N}}

\newtheorem{thm}{Theorem}
\newtheorem*{thmPdt}{Theorem \ref{thmFrechetAlg}'}
\newtheorem*{thmSubs}{Theorem \ref{thmSubst}'}
\newtheorem*{thmTFI}{Theorem \ref{thmImplicit}'}
\newtheorem*{thmGrp}{Theorem \ref{thmGroup}'}
\newtheorem*{thmS}{Theorem \ref{thmboundconvga}'}

\newtheorem{prop}{Proposition}[section]

\newtheorem{lem}[prop]{Lemma}%[section]

\newtheorem{Def}[prop]{Definition}%[section]

\newtheorem{nota}[prop]{Notation}

\theoremstyle{definition}
\newtheorem{rem}[prop]{Remark}%[section]

\newcounter{parag}[section]

\newcounter{parage}%[section]

\newcounter{paraga}

\def\sst#1{^{[#1]}}

\DeclarePairedDelimiter\abs{\lvert}{\rvert}%
\DeclarePairedDelimiter\norm{\lVert}{\rVert}%
\DeclarePairedDelimiter\croc{\langle}{\rangle}%

\makeatletter
\let\oldabs\abs
\def\abs{\@ifstar{\oldabs}{\oldabs*}}
\let\oldnorm\norm
\def\norm{\@ifstar{\oldnorm}{\oldnorm*}}
\let\oldcroc\croc
\def\croc{\@ifstar{\oldcroc}{\oldcroc*}}
\makeatother
\def\om{\omega}

\newcommand{\htab}[1]{{\stackrel{\raisebox{.23ex}{$\scriptstyle\wedge$}}{#1}}}

\def\Bbibitem#1#2{\bibitem[#1]{#2}}

\newcommand{\defeq}{\coloneqq} % \newcommand{\defeq}{:=}

\newcommand{\col}{\colon\thinspace}          %% for a map f \col A \to B

\newcommand{\zcz}{z\ii\C[[z\ii]]}

\newcommand{\gD}{\mathscr D}       %% calligraphie
\newcommand{\gE}{\mathscr E}       %% calligraphie
\newcommand{\gG}{\mathscr G}       %% calligraphie
\newcommand{\gH}{\mathscr H}       %% calligraphie
\newcommand{\gJ}{\mathscr J}       %% calligraphie
\newcommand{\gO}{\mathscr O}       %% calligraphie
\newcommand{\gP}{\mathscr P}       %% calligraphie
\newcommand{\gR}{\mathscr R}       %% calligraphie
\newcommand{\gS}{\mathscr S}       %% calligraphie
\newcommand{\gU}{\mathscr U}       %% calligraphie
\newcommand{\eith}{\ee^{\I\th}}

\newcommand{\eul}{^{\raisebox{-.23ex}{$\scriptstyle\mathrm E$}}}

\newcommand{\Stir}{^{\raisebox{-.23ex}{$\scriptstyle\mathrm S$}}}

\newcommand{\beglabel}[1]{\begin{equation}	\label{#1}}
\newcommand{\elabel}{\end{equation}}

\begin{document}

\title{Nonlinear analysis with resurgent functions}

\author{David Sauzin}

%%%%%%%%%%%%%%%%%%%%%%%%%%%%%%%%%%%%%%%%%%%%%%%%%%%%%%%%
%%%%%%%%%%%%%%%%%%%%%%%%%%%%%%%%%%%%%%%%%%%%%%%%%%%%%%%%

\maketitle

\vspace{-.75cm}

\begin{abstract}
  We provide estimates for the convolution product of an arbitrary
  number of ``resurgent functions'', that is holomorphic germs at the
  origin of~$\C$ that admit analytic continuation outside a closed
  discrete subset of~$\C$ which is stable under addition.
  Such estimates are then used to perform nonlinear operations like
  substitution in a convergent series, composition or functional
  inversion with resurgent functions, and to justify the rules of
  ``alien calculus''; they also yield implicitly defined resurgent
  functions.
The same nonlinear operations can be performed in the framework of
Borel-Laplace summability.
\end{abstract}

%%%%%%%%%%%%%%%%%%%%%%%%%%%%%%%%%%%%%%%%%%%%%%%%%%%%%
%%%%%%%%%%%%%%%%%%%%%%%%%%%%%%%%%%%%%%%%%%%%%%%%%%%%%

\section{Introduction}

%%%%%%%%%%%%%%%%%%%%%%%%%%%%%%%%%%%%%%%%%%%%%%%%%%%%%
%%%%%%%%%%%%%%%%%%%%%%%%%%%%%%%%%%%%%%%%%%%%%%%%%%%%%

In the 1980s, to deal with local analytic problems of classification of dynamical
systems, J.~\'Ecalle started to develop his theory of resurgent
functions and alien derivatives
\cite{Eca81}, \cite{Eca84}, \cite{Eca93},
which is an efficient tool for dealing with divergent series arising from
complex dynamical systems or WKB expansions, analytic invariants of differential
or difference equations, linear and nonlinear Stokes phenomena
\cite{Mal82}, \cite{Mal85}, \cite{dulac}, \cite{CNP},
\cite{DiDePh}, \cite{Bal94}, \cite{DelabP},
\cite{GelfS}, \cite{OSS}, \cite{kokyu}, \cite{bookCostin}, 
\cite{mouldSN}, \cite{KKKT}, \cite{LRR}, \cite{FS11}, \cite{gazRamis}, \cite{KKK},
\cite{EVinva}, \cite{EVinvb};
connections were also recently found with Painlev\'e asymptotics
\cite{PainlI}, 
Quantum Topology \cite{Garou08}, \cite{CGKZ}
and Wall Crossing \cite{KSWC}.

The starting point in \'Ecalle's theory is the definition of certain
subalgebras of the algebra of formal power series by means of the
formal Borel transform
\beglabel{eqdefcB}
\cB \col 
\ti\ph(z) = \sum_{n=0}^\infty a_n z^{-n-1} \in \zcz
\mapsto 
\hat\ph(\ze) = \sum_{n=0}^\infty a_n \frac{\ze^n}{n!} \in \C[[\ze]]
\elabel
(using negative power expansions in the \lhs\ and changing the name of the
indeterminate from~$z$ to~$\ze$ are just convenient conventions).

It turns out that, for a lot of interesting functional equations, one
can find formal solutions which are divergent for all~$z$ and whose
Borel transforms define holomorphic germs at~$0$ with particular
properties of analytic continuation.
The simplest examples are the Euler series
\cite{CNP}, \cite{gazRamis}, 
which can be written $\ti\ph\eul(z) = \sum_{n=0}^\infty (-1)^n n! z^{-n-1}$ and solves
a first-order linear non-homogeneous differential equation,
and the Stirling series \cite[Vol.~3]{Eca81}
\[
\ti\ph\Stir(z) = \sum_{k=1}^\infty \frac{B_{2k}}{2k(2k-1)} z^{-2k+1}
\] 
(here expressed in terms of the Bernoulli numbers),
solution of a linear non-homogeneous difference equation derived from the
functional equation for Euler's Gamma function by taking logarithms.
In both examples the Borel transform gives rise to convergent series with a
meromorphic extension to the $\ze$-plane, namely $(1+\ze)\ii$ for the Euler
series and 
$\ze^{-2}\left( \frac{\ze}{2}\coth\frac{\ze}{2} - 1 \right)$
for the Stirling series (see \cite{introsummaresur}).
In fact, holomorphic germs at~$0$ with meromorphic or algebraic
analytic continuation are examples of ``resurgent functions'';
%
%The single-valuedness of the analytic continuation must not be... 
%
% (due to nonlinearities in the problem they originate from)...
%
more generally, what is required for a resurgent function is
%
% the property which persists in many interesting situations is rather 
%
the possibility of following the analytic continuation without
encountering natural barriers.

One is thus led to distinguish certain subspaces~$\hat\gR$
of~$\C\{\ze\}$, characterized by properties of analytic
continuation which ensure a locally discrete set of singularities for
each of its members (and which do not preclude multiple-valuedness of
the analytic continuation), and to consider
\[
\ti\gR \defeq \C \oplus \cB\ii(\hat\gR) \subset \C[[z\ii]].
\]
Typically one has the strict inclusion 
$\C\{z\ii\} \subsetneq \ti\gR$ but the divergent series in~$\ti\gR$ can be
``summed'' by means of Borel-Laplace summation.
The formal series in~$\ti\gR$ as well as the holomorphic functions whose germ
at~$0$ belongs to~$\hat\gR$ are termed ``resurgent''.
(One also defines, for each $\om\in\C^*$, an ``alien operator'' which
measures the singularities at~$\om$ of certain branches of the
analytic continuation of~$\hat\ph$.)

%%%%%%%%%%%%%%%%%%%%%%%%%%%%%%%%%%%%%%%%%%%%%%%%%%%%%%%%

Later we shall be more specific about the definition of~$\hat\gR$.
This article is concerned with the convolution of resurgent functions:
the convolution in $\C\{\ze\}$ is the commutative associative
product defined by
\beglabel{eqdefconvoldeux}
\hat\ph_1 * \hat\ph_2(\ze) =
\int_0^\ze \hat\ph_1(\ze_1) \hat\ph_2(\ze-\ze_1) \,\dd\ze_1
\qquad \text{for $|\ze|$ small enough,}
\elabel
for any $\hat\ph_1,\hat\ph_2 \in \C\{\ze\}$,
which reflects the Cauchy product of formal series via the formal Borel
transform:
\[
\cB\ti\ph_1 = \hat\ph_1 \enspace\text{and}\enspace 
\cB\ti\ph_2 = \hat\ph_2 
\quad \Longrightarrow \quad
\cB(\ti\ph_1 \ti\ph_2) = \hat\ph_1 * \hat\ph_2.
\]
Since the theory was designed to deal with nonlinear problems, it is of
fundamental importance to control the convolution product of resurgent
functions;
however, this requires to follow the analytic continuation of the function
defined by~\eqref{eqdefconvoldeux}, which turns out not to be an easy task.
In fact, probably the greatest difficulties in understanding and applying
resurgence theory are connected with the problem of controlling the analytic
continuation of functions defined by such integrals or by analogous multiple
integrals.
Even the mere stability under convolution of the spaces~$\hat\gR$ is not obvious
\cite{Eca81}, \cite{CNP}, \cite{Y_Ou}, \cite{stabiconv}.

%%%%%%%%%%%%%%%%%%%%%%%%%%%%%%%%%%%%%%%%%%%%%%%%%%%%%%%%

We thus need to estimate the convolution product of two or more resurgent
functions, both for concrete manipulations of resurgent functions in nonlinear
contexts and for the foundations of the resurgence theory.
For instance, such estimates will allow us to check that, when we come back to
the resurgent series via~$\cB$, the exponential of a resurgent series is
resurgent and that more generally one can substitute resurgent series in
convergent power expansions, or define implicitly a resurgent series,
or develop ``alien calculus'' when manipulating \'Ecalle's alien derivatives.
They will also show that the group of ``formal tangent-to-identity
diffeomorphisms at~$\infty$'', \ie the group (for the composition law)
$z+\C[[z\ii]]$, admits $z+\ti\gR$ as a subgroup, which is particularly useful
for the study of holomorphic tangent-to-identity diffeomorphisms~$f$ (in this
classical problem of local holomorphic dynamics \cite{Milnor}, the Fatou
coordinates have the same resurgent asymptotic expansion, the so-called direct
iterator $f^* \in z+\ti\gR$ of \cite{Eca81}; thus its inverse, the inverse
iterator, also belongs to $z+\ti\gR$, as well as its exponential, which appears
in the Bridge equation connected with the ``horn maps''---see
\S~\ref{secgroupresurdiffeos}).

%%%%%%%%%%%%%%%%%%%%%%%%%%%%%%%%%%%%%%%%%%%%%%%%%%%%%%%%

Such results of stability of the algebra of resurgent series under nonlinear
operations are mentioned in \'Ecalle's works, however the arguments there are
sketchy and it was desirable to provide a proof.\footnote{%
  This was one of the tasks undertaken in the seminal book \cite{CNP}
  but, despite its merits, one cannot say that this book clearly
  settled this particular issue: the proof of the estimates for the
  convolution is obscure and certainly contains at least one mistake
  (see Remark~\ref{remmistCNP}).
}
Indeed, the subsequent authors dealing with resurgent series either
took such results for granted or simply avoided resorting to them.
The purpose of this article is to give clear statements with rigorous
and complete proofs, so as to clarify the issue and contribute to make
resurgence theory more accessible, hopefully opening the way for
new applications of this powerful theory.

%%%%%%%%%%%%%%%%%%%%%%%%%%%%%%%%%%%%%%%%%%%%%%%%%%%%%%%%

In this article, we shall deal with a particular case of resurgence called $\Om$-continuability
or $\Om$-resurgence, which means that we fix in advance a discrete subset~$\Om$
of~$\C$ and restrict ourselves to those resurgent functions whose analytic
continuations have no singular point outside of~$\Om$.
Many interesting cases are already covered by this definition (one encounters
$\Om$-continuable germs with $\Om=\Z$ when dealing with differential equations
formally conjugate to the Euler equation or in the study of the saddle-node
singularities \cite{Eca84}, \cite{mouldSN}, or with $\Om = 2\pi\I\Z$ when
dealing with certain difference equations like Abel's equation for
tangent-to-identity diffeomorphisms \cite{Eca81}, \cite{kokyu}, \cite{EVinva}).
We preferred to restrict ourselves to this situation so as to make our method
more transparent, even if more general definitions of resurgence can be
handled---see Section~\ref{secPossAppl}.
An outline of the article is as follows: 
\begin{enumerate}[--]
\item
In Section~\ref{secOmcontg}, we recall the precise definition of the
corresponding algebras of resurgent functions, denoted by~$\hat\gR_\Om$, and
state Theorem~\ref{thmboundconvga}, which is our main result on the control of
the convolution product of an arbitrary number of $\Om$-continuable functions.
\item
In Section~\ref{secNLappli}, we give applications to the construction of a
Fr\'echet algebra structure on~$\ti\gR_\Om$ (Theorem~\ref{thmFrechetAlg})
and to the stability of $\Om$-resurgent series under substitution (Theorem~\ref{thmSubst}), 
implicit function (Theorem~\ref{thmImplicit})
and composition (Theorem~\ref{thmGroup});
we also mention other possible applications and similar results for
$1$-summable series.
\item 
The proof of Theorem~\ref{thmboundconvga} is given in
Sections~\ref{secstartpf}--\ref{secestimisot}.
\item
Finally, there is an appendix on a few facts of the theory of currents which are
used in the proof of the main theorem.
\end{enumerate}

Our method consists in representing the analytic continuation of a convolution
product as the integral of a holomorphic $n$-form on a singular $n$-simplex
obtained as a suitable deformation of the standard $n$-simplex;
we explain in Sections~\ref{secstartpf}--\ref{secdefor} what kind of
deformations (``adapted origin-fixing isotopies'' of the identity) are licit in
order to provide the analytic continuation and how to produce them.
We found the theory of currents very convenient to deal with our integrals of
holomorphic forms, because it allowed us to content ourselves with little
regularity: 
the deformations we use are only Lipschitz continuous, because they are built from
the flow of non-autonomous Lipschitz vector fields---see
Section~\ref{secConstrIsot}.
Section~\ref{secestimisot} contains the last part of the proof, which consists
in providing appropriate estimates. 

\newpage

%%%%%%%%%%%%%%%%%%%%%%%%%%%%%%%%%%%%%%%%%%%%%%%%%%%%%
%%%%%%%%%%%%%%%%%%%%%%%%%%%%%%%%%%%%%%%%%%%%%%%%%%%%%

\section{The convolution of $\Om$-continuable germs}	\label{secOmcontg}

%%%%%%%%%%%%%%%%%%%%%%%%%%%%%%%%%%%%%%%%%%%%%%%%%%%%%
%%%%%%%%%%%%%%%%%%%%%%%%%%%%%%%%%%%%%%%%%%%%%%%%%%%%%

\begin{nota}
For any $R>0$ and $\ze_0\in\C$ we use the notations
$D(\ze_0,R) \defeq \{\, \ze\in\C \mid |\ze-\ze_0|< R \,\}$,
$\D_R \defeq D(0,R)$, 
$\D^*_R \defeq \D_R \setminus \{0\}$.
%
% We call ``path'' any continuous piecewise $C^1$ function $\ga \col J \to\C$,
% where $J=[a,b]$ is a compact interval of~$\R$.
%
% if a holomorphic germ~$\hat\ph$ at~$\ga(a)$ admits analytic continuation
% along~$\ga$, then we denote by $\cont_\ga\hat\ph$ the resulting
% holomorphic germ at~$\ga(b)$;
% %
% for $c\in [a,b]$, we denote by $\ga|c$ the truncated path $\ga_{|[a,c]}$, thus
% $\cont_{\ga|c}\hat\ph$ is a germ at~$\ga(c)$.
%
\end{nota}

%%%%%%%%%%%%%%%%%%%%%%%%%%%%%%%%%%%%%%%%%%%%%%%%%%%%%

Let $\Om$ be a closed, discrete subset of~$\C$ containing~$0$.
We set 
\[
\rho(\Om) \defeq  \min\big\{ |\om|, \; \om\in\Om\setminus\{0\} \big\}.
\]
Recall \cite{stabiconv} that the space $\hat\gR_\Om$ of all $\Om$-continuable
germs is the subspace of $\C\{\ze\}$ which can be defined by the fact that, for
arbitrary $\ze_0 \in \D_{\rho(\Om)}$,
\[
\hat\ph\in\hat\gR_\Om \quad\Longleftrightarrow\quad
\left| \begin{aligned}
&\text{$\hat\ph$ germ of holomorphic function of~$\D_{\rho(\Om)}$ admitting analytic continuation}\\
&\text{along any path $\ga\col [0,1] \to \C$ such that $\ga(0) = \ze_0$ and $\ga\big( (0,1] \big)
\subset \C\setminus\Om$.}
\end{aligned} \right.
\]
For example, for the Euler series, \resp the Stirling series, the Borel
transform belongs to $\hat\gR_\Om$ as soon as $1 \in \Om$, 
\resp ${2\pi\I\Z^*} \subset \Om$.

It is convenient to rephrase the property of $\Om$-continuability as holomorphy
on a certain Riemann surface spread over the complex plane, 
$(\gS_\Om,\pi_\Om)$.

%%%%%%%%%%%%%%%%%%%%%%%%%%%%%%%%%%%%%%%%%%%%%%%%%%%%%
\begin{Def}
Let $I \defeq [0,1]$ and consider the set~$\gP_\Om$ of all paths $\ga \col I \to \C$ such that 
either $\ga(I) = \{0\}$
or $\ga(0)=0$ and $\ga\big( (0,1] \big) \subset \C\setminus\Om$.
We denote by
\[
\gS_\Om \defeq \gP_\Om / \sim
\]
the quotient set of~$\gP_\Om$ by the equivalence relation~$\sim$
defined by
\[
\ga \sim \ga' \quad\Longleftrightarrow\quad
\exists (\ga_s)_{s\in I} \;\text{such that}\ens
\left\{ \begin{aligned}
&\text{for each $s\in I$, $\ga_s\in\gP_\Om$ and $\ga_s(1)=\ga(1)$}\\
& \text{$(s,t)\in I\times I \mapsto \ga_s(t)\in\C$ is continuous, $\ga_0=\ga$, $\ga_1 = \ga'$}
\end{aligned} \right.
\]
for $\ga,\ga' \in \gP_\Om$ (homotopy with fixed endpoints).
The map $\ga\in\gP_\Om \mapsto \ga(1)\in \{0\}\cup\C\setminus\Om$ passes to the
quotient and defines the ``projection'' 
\beglabel{eqdefpiOm}
\pi_\Om \col \ze \in \gS_\Om \to \dze \in \{0\}\cup\C\setminus\Om.
\elabel
We equip~$\gS_\Om$ with the unique structure of Riemann surface which turns~$\pi_\Om$
into a local biholomorphism.
The equivalence class of the trivial path $\ga(t)\equiv0$ is denoted by~$0_\Om$
and called the origin of~$\gS_\Om$. 
\end{Def}
%%%%%%%%%%%%%%%%%%%%%%%%%%%%%%%%%%%%%%%%%%%%%%%%%%%%%

We obtain a connected, simply connected Riemann surface~$\gS_\Om$, which is
somewhat analogous to the universal cover of $\C\setminus\Om$ except for the
special role played by~$0$ and~$0_\Om$:
since we assumed $0\in\Om$, the equivalence class~$0_\Om$ of the
trivial path is reduced to the trivial path and is the only point of~$\gS_\Om$ which projects onto~$0$.
It belongs to the \emph{principal sheet of~$\gS_\Om$}, defined as the set of
all $\ze\in\gS_\Om$ which can be represented by a line segment (\ie such that the path
$t\in[0,1] \mapsto t\, \dze$ belongs to $\gP_\Om$ and
represents~$\ze$);
observe that~$\pi_\Om$ induces a biholomorphism from the principal
sheet of~$\gS_\Om$ to the cut plane $\dst U_\Om \defeq\C \setminus
\bigcup_{\om\in\Om\setminus\{0\}} \om[1,+\infty)$.

%%%%%%%%%%%%%%%%%%%%%%%%%%%%%%%%%%%%%%%%%%%%%%%%%%%%%

Any holomorphic function of~$\gS_\Om$ identifies itself with a convergent germ
at the origin of~$\C$ which admits analytic continuation along all the paths
of~$\gP_\Om$, so that
\[
\hat\gR_\Om \simeq \gO(\gS_\Om)
\]
(see \cite{Eca81}, \cite{kokyu}). 
We usually use the same symbol~$\hat\ph$ for a function of~$\gO(\gS_\Om)$ or the
corresponding germ of holomorphic function at~$0$ (\ie its Taylor series).
Each $\hat\ph \in \hat\gR_\Om$ has a well-defined principal branch
holomorphic in~$U_\Om$, obtained (via~$\pi_\Om$) by restriction to the
principal sheet of~$\gS_\Om$, for which~$0$ is a regular point,
but the points of~$\gS_\Om$ which lie outside of the principal sheet
correspond to branches of the analytic continuation which might have a
singularity at~$0$
(for instance, as soon as $\{0,1\} \subset \Om$, the Taylor series
$\sum_{n\ge0} \frac{\ze^n}{n+1} = -\frac{1}{\ze}\log(1-\ze)$ defines a
member of~$\hat\gR_\Om$ of which all branches except the principal one
have a simple pole at~$0$).

%%%%%%%%%%%%%%%%%%%%%%%%%%%%%%%%%%%%%%%%%%%%%%%%%%%%%

\emph{From now on we assume that $\Om$ is stable under addition.} 
According to \cite{stabiconv}, this ensures that $\hat\gR_\Om$ is stable under convolution.
Our aim is to provide explicit bounds for the analytic continuation of a
convolution product of two or more factors belonging
to~$\hat\gR_\Om$.

It is well-known that, if $U\subset\{0\}\cup\C\setminus\Om$ is open and
star-shaped \wrt~$0$ (as is~$U_\Om$)
and two functions $\hat\ph_1,\hat\ph_2$ are holomorphic in~$U$, 
then their convolution product has an analytic continuation to~$U$ which is
given by the very same formula~\eqref{eqdefconvoldeux}; by induction, one gets a
representation of a product of $n$ factors $\hat\ph_j\in\gO(U)$ as an iterated
integral,
$\hat\ph_1 * \cdots * \hat\ph_n(\ze) = $
\beglabel{eqiterconvol}
\int_0^\ze \dd \ze_1 \int_0^{\ze-\ze_1} \dd \ze_2 \cdots
\int_0^{\ze-(\ze_1+\cdots+\ze_{n-2})} \dd \ze_{n-1} \,
\hat\ph_1(\ze_1) \cdots \hat\ph_{n-1}(\ze_{n-1})
\hat\ph_n( \ze - (\ze_1+\cdots+\ze_{n-1}) )
\elabel
for any $\ze \in U$, which leads to
\beglabel{ineqppalsheet}
\abs{ \hat\ph_1*\cdots*\hat\ph_n(\ze) } \le 
\frac{\abs{\ze}^{n-1}}{(n-1)!} 
\max_{[0,\ze]}\abs{\hat\ph_1} \cdots \max_{[0,\ze]}\abs{\hat\ph_n},
\qquad \ze\in U.
\elabel
This allows one to control convolution products in the principal sheet
of~$\gS_\Om$ 
(which is already sufficient to deal with $1$-summability issues---see
Section~\ref{secNLsumma})
but, to reach the other sheets, formula~\eqref{eqdefconvoldeux}
must be replaced by something else, as explained \eg in \cite{stabiconv}. What
about the bounds for a product of $n$ factors then?
To state our main result, we introduce
\begin{nota}
The function $R_\Om \col \gS_\Om \to (0,+\infty)$ is defined by
\beglabel{eqdefRze}
\ze\in\gS_\Om \;\mapsto\; R_\Om(\ze) \defeq \begin{dcases*}
\dist\big(\dze,\Om\setminus\{0\}\big) 
& if $\ze$ belongs to the principal sheet of $\gS_\Om$ \\[1ex]
\dist\big(\dze,\Om\big) 
& if not
\end{dcases*}
\elabel
(where $\dze$ is the shorthand for $\pi_\Om(\ze)$ defined by~\eqref{eqdefpiOm}).
For $\de,L>0$, we set
\begin{multline}	\label{eqdefLdeL}
\cK_{\de,L}(\Om) \defeq \big\{\, 
\ze\in\gS_\Om \mid \text{$\exists \ga$ path of~$\gS_\Om$ 
with endpoints $0_\Om$ and~$\ze$, of length $\le L$,}\\[1ex]
\text{such that $R_\Om(\ga(t)) \ge \de$ for all~$t$}
\,\}.
\end{multline}
\end{nota}

Informally, $\cK_{\de,L}(\Om)$ consists of the points of~$\gS_\Om$ which can be
joined to~$0_\Om$ by a path of length $\le L$ ``staying at distance $\ge\de$ from
the boundary''.\footnote{
Given $\ze\in\gS_\Om$, observe that any $\hat\ph \in \hat\gR_\Om$
induces a function holomorphic in $D\big(\dze,R_\Om(\ze)\big)$ and
$R_\Om(\ze)$ is maximal for that property.
  The idea is that~$R_\Om$ measures the distance to the closest
  possibly singular point, \ie the distance to~$\Om$ except that on
  the principal sheet $0$ must not be considered as a possibly
  singular point.
}
Observe that $\big(\cK_{\de,L}(\Om)\big)_{\de,L>0}$ is an exhaustion of~$\gS_\Om$ by
compact subsets.
If $L+\de<{\rho(\Om)}$, then $\cK_{\de,L}(\Om)$ is just the lift of the closed
disc~$\ov\D_L$ in the principal sheet of~$\gS_\Om$.

%%%%%%%%%%%%%%%%%%%%%%%%%%%%%%%%%%%%%%%%%%%%%%%%%%%%%

\label{secmainthm}
\begin{thm}	\label{thmboundconvga}
Let $\Om\subset\C$ be closed, discrete, stable under addition, with $0\in\Om$.
Let $\de,L>0$ with $\de < \rho(\Om)$ and 
\beglabel{eqdefconstants}
C \defeq \rho(\Om)\,\ee^{3+6L/\de}, \qquad
\de' \defeq \demi \rho(\Om)\,\ee^{-2-4L/\de}, \qquad
L' \defeq L+\frac{\de}{2}.
\elabel
Then, for any $n\ge1$ and $\hat\ph_1,\ldots,\hat\ph_n \in \hat\gR_\Om$,
\beglabel{ineqconvga}
\max_{\cK_{\de,L}(\Om)}\abs{ \hat\ph_1*\cdots*\hat\ph_n }
\le \frac{2}{\de} \cdot \frac{C^n}{n!} \cdot
\max_{\cK_{\de',L'}(\Om)}\abs{\hat\ph_1} \cdots \max_{\cK_{\de',L'}(\Om)}\abs{\hat\ph_n}.
\elabel
\end{thm} 

%%%%%%%%%%%%%%%%%%%%%%%%%%%%%%%%%%%%%%%%%%%%%%%%%%%%%

The proof of Theorem~\ref{thmboundconvga} will start in
Section~\ref{secstartpf}.
We emphasize that $\de,\de',L,L',C$ do not depend on~$n$, which is important in
applications.

%%%%%%%%%%%%%%%%%%%%%%%%%%%%%%%%%%%%%%%%%%%%%%%%%%%%%

\begin{rem}
In fact, \emph{a posteriori}, one can remove the assumption $0 \in \Om$.
Suppose indeed that $\Om$ is a non-empty closed discrete subset of~$\C$ which does \emph{not}
contain~$0$. 
Defining the space~$\hat\gR_\Om$ of $\Om$-continuable germs as above
\cite{stabiconv},
we then get $\hat\gR_\Om \simeq \gO(\gS_\Om)$, where $\gS_\Om$ is the universal
cover of $\C\setminus\Om$ with base point at the origin
(the fibre of~$0$ is no longer exceptional).
Clearly $\hat\gR_\Om \subset \hat\gR_{\{0\}\cup\Om}$, but the
inclusion is strict, 
because $\Om$-continuable germs are required to extend analytically
through~$0$ even when following a path which has turned around the
points of~$\Om$ and \eg
$\sum_{n\ge0} \frac{\ze^n}{(n+1)\om^{n+1}} = -\frac{1}{\ze}\log(1-\frac{\ze}{\om})$
is in $\hat\gR_{\{0\}\cup\Om}$ but not in $\hat\gR_\Om$ for any $\om\in\Om$.
Suppose moreover that $\Om$ is stable under addition.
It is shown in \cite{stabiconv} that also in this case is $\hat\gR_\Om$ stable
under convolution.
\emph{One can adapt all the results of this article to this case.}
It is sufficient to observe that any point~$\ze$ of~$\gS_\Om$ can be defined by a path
$\ga \col [0,1] \to \C$ such that $\ga(0) \in \D_{\rho(\Om)}$,
$\ga\big((0,1)\big) \cap \big(\Om\cup\{0\}\big) = \emptyset$ and
$\ga(1) \notin\Om$;
if $\ga(1) \neq 0$, then the situation is explicitly covered by this article;
if $\ga(1) = 0$, then we can still apply our results to the neighbourhing points
and make use of the maximum principle.
\end{rem}

%%%%%%%%%%%%%%%%%%%%%%%%%%%%%%%%%%%%%%%%%%%%%%%%%%%%%
%%%%%%%%%%%%%%%%%%%%%%%%%%%%%%%%%%%%%%%%%%%%%%%%%%%%%

\section{Application to nonlinear operations with $\Om$-resurgent series}
\label{secNLappli}

%%%%%%%%%%%%%%%%%%%%%%%%%%%%%%%%%%%%%%%%%%%%%%%%%%%%%
%%%%%%%%%%%%%%%%%%%%%%%%%%%%%%%%%%%%%%%%%%%%%%%%%%%%%

\subsection{Fr\'echet algebra structure on~$\ti\gR_\Om$}

%%%%%%%%%%%%%%%%%%%%%%%%%%%%%%%%%%%%%%%%%%%%%%%%%%%%%
%%%%%%%%%%%%%%%%%%%%%%%%%%%%%%%%%%%%%%%%%%%%%%%%%%%%%

Recall that $\Om$ is a closed discrete subset of~$\C$ which contains~$0$ and is stable
under addition.
The space of $\Om$-resurgent series is 
\[
\ti\gR_\Om = \C \oplus \cB\ii(\hat\gR_\Om).
\]
As a vector space, it is isomorphic to $\C\times\gO(\gS_\Om)$.
We now define seminorms on~$\ti\gR_\Om$ which will ease the exposition.

%%%%%%%%%%%%%%%%%%%%%%%%%%%%%%%%%%%%%%%%%%%%%%%%%%%%%

\begin{Def}
Let $K \subset \gS_\Om$ be compact.
%
% Let $\cK = \cK_{\de,L}(\Om)$ be any of the compact subsets of~$\gS_\Om$ 
% defined by~\eqref{eqdefLdeL}.
%
We define the seminorm $\norm{\,\cdot\,}_K \col \ti\gR_\Om \to \R^+$ by
\[
\ti\phi \in \ti\gR_\Om \mapsto
\norm*{\ti\phi}_K \defeq \abs{c} + \max_K\abs{\hat\ph},
\]
where $\ti\phi = c + \cB\ii\hat\ph$, $c\in \C$, $\hat\ph \in \hat\gR_\Om$.
\end{Def}

%%%%%%%%%%%%%%%%%%%%%%%%%%%%%%%%%%%%%%%%%%%%%%%%%%%%%

Choosing $K_N = \cK_{\de_N,L_N}(\Om)$, $N\in\N^*$, with any pair of sequences
$\de_N \downarrow 0$ and $L_N \uparrow\infty$
(so that $\gS_\Om$ is the increasing union of the compact sets~$K_N$),
we get a countable family of seminorms which defines a structure of Fr\'echet
space on~$\ti\gR_\Om$.
A direct consequence of Theorem~\ref{thmboundconvga} is the continuity of the
Cauchy product for this Fr\'echet structure. More precisely:

%%%%%%%%%%%%%%%%%%%%%%%%%%%%%%%%%%%%%%%%%%%%%%%%%%%%%

\begin{thm}	\label{thmFrechetAlg}
For any~$K$ there exist $K'\supset K$ and $C>0$ such that, for any $n\ge r\ge0$,
\beglabel{ineqGenProdrwct}
\norm*{\ti\phi_1 \cdots \ti\phi_n}_K \le \frac{C^n }{r!}
\norm*{\ti\phi_1}_{K'} \cdots \norm*{\ti\phi_n}_{K'}
\elabel
for every sequence $(\ti\phi_1,\ldots,\ti\phi_n)$ of $\Om$-resurgent series, $r$
of which have no constant term.

In particular, $\ti\gR_\Om$ is a Fr\'echet algebra.
\end{thm}

%%%%%%%%%%%%%%%%%%%%%%%%%%%%%%%%%%%%%%%%%%%%%%%%%%%%%

\begin{proof}
Let us fix~$K$ compact and choose $\de,L>0$ so that $K \subset \cK_{\de,L}(\Om)$.
Let $\de',L'$ be as in~\eqref{eqdefconstants} and $K' \defeq
\cK_{\de',L'}(\Om)$.
According to Theorem~\ref{thmboundconvga}, we can choose $C \ge 1$ large enough so that
for any $m\ge1$ and $\ti\ph_1,\ldots,\ti\ph_m \in \cB\ii(\hat\gR_\Om)$, 
\beglabel{ineqWtcph} 
\norm*{\ti\ph_1\cdots\ti\ph_m}_K \le \frac{C^m}{m!} 
\norm*{\ti\ph_1}_{K'} \cdots \norm*{\ti\ph_m}_{K'}.
\elabel

Let $n\ge r$ and and $s\defeq n-r$. 
Given $n$ resurgent series among which~$r$ have no constant term, we can label
them so that 
\[
\ti\phi_1 = c_1 + \ti\ph_1,\; \ldots,\; \ti\phi_s = c_s + \ti\ph_s, \; \;
\ti\phi_{s+1} = \ti\ph_{s+1},\; \ldots,\;\ti\phi_n = \ti\ph_n,
\]
with 
$c_1,\ldots,c_s \in \C$ and $\ti\ph_1,\ldots,\ti\ph_n \in \cB\ii(\hat\gR_\Om)$.
Then $\ti\phi_1 \cdots \ti\phi_n = c + \ti\psi$ with
\[
\ti\psi = \sum_I
c_{i_1} \cdots c_{i_p} 
\ti\ph_{j_1} \cdots \ti\ph_{j_q}
\ti\ph_{s+1} \cdots \ti\ph_{n}
\in \cB\ii(\hat\gR_\Om),
\]
where 
either $r\ge1$, $c=0$ and the summation is over all subsets 
$I = \{ i_1, \ldots, i_p \}$ of $\{1,\ldots,s\}$ (of any cardinality~$p$), with 
$\{j_1,\ldots,j_q\} \defeq \{1,\ldots,s\}\setminus I$,
or $r=0$, $c = c_1 \cdots c_n$ and the summation is restricted to the
proper subsets of $\{1,\ldots,n\}$.
Using inequality~\eqref{ineqWtcph}, we get
$\norm*{\ti\phi_1 \cdots \ti\phi_n}_K \le$
\[
\sum_I \frac{C^{q+r}}{(q+r)!} \, \abs*{c_{i_1} \cdots c_{i_p}} \,
\norm*{\ti\ph_{j_1}}_{K'} \cdots \norm*{\ti\ph_{j_q}}_{K'}
\norm*{\ti\ph_{s+1}}_{K'} \cdots \norm*{\ti\ph_{n}}_{K'} 
\le \frac{C^n}{r!} \norm*{\ti\phi_1}_{K'} \cdots \norm*{\ti\phi_n}_{K'}
\]
(even if $r=0$, in which case we include $I=\{1,\ldots,n\}$ in the
summation and use $C\ge1$).

The continuity of the multiplication in~$\ti\gR_\Om$ follows, as a particular
case when $n=2$.
\end{proof}

%%%%%%%%%%%%%%%%%%%%%%%%%%%%%%%%%%%%%%%%%%%%%%%%%%%%%

\begin{rem}	\label{remactionddzcB}
$\ti\gR_\Om$ is even a differential Fr\'echet algebra since $\frac{\dd\,}{\dd z}$
induces a continuous derivation of~$\ti\gR_\Om$. 
Indeed, the very definition of~$\cB$ in~\eqref{eqdefcB} shows that
\[
\ti\phi = c + \cB\ii\hat\ph 
\quad\Longrightarrow\quad
\frac{\dd\ti\phi}{\dd z} = \cB\ii\hat\psi 
\ens\text{with}\; \hat\psi(\ze) = -\ze\hat\ph(\ze),
\]
whence $\norm*{\frac{\dd\ti\phi}{\dd z}}_K \le D(K) \norm*{\ti\phi}_K$
with $D(K) = \max_{\ze\in K} \abs{\ze}$.
\end{rem}

%%%%%%%%%%%%%%%%%%%%%%%%%%%%%%%%%%%%%%%%%%%%%%%%%%%%%
%%%%%%%%%%%%%%%%%%%%%%%%%%%%%%%%%%%%%%%%%%%%%%%%%%%%%

\subsection{Substitution and implicit resurgent functions}

%%%%%%%%%%%%%%%%%%%%%%%%%%%%%%%%%%%%%%%%%%%%%%%%%%%%%
%%%%%%%%%%%%%%%%%%%%%%%%%%%%%%%%%%%%%%%%%%%%%%%%%%%%%

\begin{Def}	\label{defCVresur}
For any $r\in\N^*$, we define $\ti\gR_\Om\{w_1,\ldots,w_r\}$ as the subspace of
$\ti\gR_\Om[[w_1,\ldots,w_r]]$ consisting of all formal power series
\[
\ti H = \sum_{\bk = (k_1,\ldots,k_r) \in \N^r} \ti H_{\bk}(z) \, w_1^{k_1} \cdots w_r^{k_r}
\]
with coefficients $\ti H_{\bk} = \ti H_{\bk}(z) \in \ti\gR_\Om$ such that,
for every compact $K\subset\gS_\Om$, there exist positive numbers $A_K,B_K$ such that
\beglabel{ineqHCVAB}
\norm*{\ti H_{\bk}}_K \le A_K \, B_K^{\abs{\bk}}
\elabel
for all $\bk \in \N^r$ (with the notation $\abs{\bk} = k_1 + \cdots + k_r$).
\end{Def}

The idea is to consider formal series ``resurgent in~$z$ and
convergent in $w_1,\ldots,w_r$''.
We now show that one can substitute resurgent series in such a convergent
series.
Observe that $\ti\gR_\Om\{w_1,\ldots,w_r\}$ can be considered as a subspace of
$\C[[z\ii,w_1,\ldots,w_r]]$.

%%%%%%%%%%%%%%%%%%%%%%%%%%%%%%%%%%%%%%%%%%%%%%%%%%%%%

\begin{thm}	\label{thmSubst}
\begin{enumerate}[(i)]
\item
The space $\ti\gR_\Om\{w_1,\ldots,w_r\}$ is a subalgebra of $\C[[z\ii,w_1,\ldots,w_r]]$.
\item
Suppose that $\ti\ph_1,\ldots,\ti\ph_r \in \ti\gR_\Om$ have no constant term.
Then for any
$\ti H = \sum \ti H_{\bk} \, w_1^{k_1} \cdots w_r^{k_r} 
\in\ti\gR_\Om\{w_1,\ldots,w_r\}$,
the series 
\[
\ti H(\ti\ph_1,\ldots,\ti\ph_r) \defeq
\sum_{\bk\in\N^r} \ti H_{\bk}\, \ti\ph_1^{k_1} \cdots \ti\ph_r^{k_r} 
\in \C[[z\ii]]
\]
is convergent in $\ti\gR_\Om$ and, for every compact $K\subset \gS_\Om$, there
exist a compact $K'\supset K$ and a constant $C>0$ so that
\[
\norm*{\ti H(\ti\ph_1,\ldots,\ti\ph_r)}_K \le C A_{K'} \,
\ee^{ C B_{K'} \big( \norm*{\ti\ph_1}_{K'} + \cdots + \norm*{\ti\ph_r}_{K'} \big) }
\]
(with notations similar to those of Definition~\ref{defCVresur} for $A_{K'},B_{K'}$).
\item
The map $\ti H \in\ti\gR_\Om\{w_1,\ldots,w_r\} \mapsto 
\ti H(\ti\ph_1,\ldots,\ti\ph_r) \in \ti\gR_\Om$
is an algebra homomorphism.
\end{enumerate}
\end{thm}

%%%%%%%%%%%%%%%%%%%%%%%%%%%%%%%%%%%%%%%%%%%%%%%%%%%%%

\begin{proof}
The proof of the first statement is left as an exercise.
Observe that the series of formal series
\[
\ti\chi = \sum_{\bk\in\N^r} \ti H_{\bk}\, \ti\ph_1^{k_1} \cdots \ti\ph_r^{k_r} 
\]
is formally convergent%
\footnote{%
A family of formal series in $\C[[z\ii]]$ is formally summable if it has only finitely many
members of order $\le N$ for every $N\in\N$. 
Notice that if a formally summable family is made up of $\Om$-resurgent series
and is summable for the semi-norms $\norm*{\cdot}_K$, then the formal sum
in~$\C[[z\ii]]$ and the sum in~$\ti\gR_\Om$ coincide 
(because the Borel transform of the formal sum is nothing but the Taylor series
at~$0$ of the Borel transform of the sum in~$\ti\gR_\Om$).
}
in $\C[[z\ii]]$, because $\ti H_{\bk}\, \ti\ph_1^{k_1} \cdots \ti\ph_r^{k_r}$
has order $\ge \abs{\bk}$; this is in fact a particular case of composition of
formal series and the fact that the map
\[ \ti H \in\ti\C[[z\ii,w_1,\ldots,w_r]] \mapsto 
\ti H(\ti\ph_1,\ldots,\ti\ph_r) \in \C[[z\ii]] \]
is an algebra homomorphism is well-known.
The last statement will thus follow from the second one.

Let us fix $K\subset \gS_\Om$ compact.
We first choose $K'$ and~$C$ as in Theorem~\ref{thmFrechetAlg}, and then $A = A_{K'}$,
$B = B_{K'}$ so that~\eqref{ineqHCVAB} holds relatively to~$K'$.
For each $\bk \in \N^r$, inequality~\eqref{ineqGenProdrwct} yields
\[
\norm*{\ti H_{\bk}\, \ti\ph_1^{k_1} \cdots \ti\ph_r^{k_r} }_K \le
\frac{C^{\abs{\bk}+1}}{\abs{\bk}!} \norm*{\ti H_{\bk}}_{K'} 
\norm*{\ti\ph_1}_{K'}^{k_1} \cdots \norm*{\ti\ph_r}_{K'}^{k_r} 
\le C A \frac{(C B)^{\abs{\bk}}}{\abs{\bk}!} 
\norm*{\ti\ph_1}_{K'}^{k_1} \cdots \norm*{\ti\ph_r}_{K'}^{k_r}
\]
and the conclusion follows easily.
\end{proof}

%%%%%%%%%%%%%%%%%%%%%%%%%%%%%%%%%%%%%%%%%%%%%%%%%%%%%

As an illustration, for $\ti\phi = c + \ti\ph$ with $c\in\C$ and $\ti\ph \in
\cB\ii(\hat\gR_\Om)$, we have
\[
\exp(\ti\phi) = \ee^c \, \sum_{n\ge0} \frac{1}{n!} \ti\ph^n \in \ti\gR_\Om
\]
and, if moreover $c\neq0$,
\[
{1}/{\ti\phi} = \sum_{n\ge0} (-1)^n c^{-n-1} \ti\ph^n \in \ti\gR_\Om.
\]

\begin{rem}
  An example of application of Theorem~\ref{thmSubst} is provided by
  the exponential of the Stirling series~$\ti\ph\Stir$ mentioned in
  the introduction:
we obtain the $2\pi\I\Z$-resurgence of the divergent series $\exp(\ti\ph\Stir)$
which, according to the refined Stirling formula, is the asymptotic expansion of
$\frac{1}{\sqrt{2\pi}} z^{\demi-z} \ee^z \, \Ga(z)$
(in fact the formal series $\exp(\ti\ph\Stir)$ is $1$-summable in the
directions of $( -\frac{\pi}{2}, \frac{\pi}{2} )$, and
this function is its Borel-Laplace sum in the sector $-\pi < \arg z < \pi$;
see Section~\ref{secNLsumma}).
\end{rem}

%%%%%%%%%%%%%%%%%%%%%%%%%%%%%%%%%%%%%%%%%%%%%%%%%%%%%

We now show an implicit function theorem for resurgent series.
\begin{thm}	\label{thmImplicit}
Let $F(x,y) \in \C[[x,y]]$ be such that $F(0,0)=0$ and $\pa_y F(0,0)\neq0$, and
call~$\ph(x)$ the unique solution in $x\C[[x]]$ of the equation
\beglabel{eqimplicitph}
F\big( x, \ph(x) \big) = 0.
\elabel
Let $\ti F(z,y) \defeq F(z\ii,y) \in \C[[z\ii,y]]$ and
$\ti\ph(z) \defeq \ph(z\ii) \in z\ii\C[[z\ii]]$, so that $\ti\ph$ is implicitly
defined by the equation $\ti F\big( z, \ti\ph(z) \big) = 0$.
Then
\[
\ti F(z,y) \in \ti\gR_\Om\{y\}
\quad\Longrightarrow\quad
\ti\ph(z) \in \ti\gR_\Om.
\]
\end{thm}

%%%%%%%%%%%%%%%%%%%%%%%%%%%%%%%%%%%%%%%%%%%%%%%%%%%%%

\begin{proof}
Without loss of generality we can assume $\pa_y F(0,0) = -1$ and write
\[
F(x,y) = -y + f(x) + R(x,y)
\]
with $f(x) = F(x,0) \in x\C[[x]]$ and a quadratic remainder
\[
R(x,y) =\sum_{n\ge1} R_n(x) y^n, \qquad
R_n(x) \in \C[[x]], \quad R_1(0) = 0.
\]

When viewed as formal transformation in~$y$, the formal series 
$\th(x,y) \defeq y - R(x,y)$ is invertible, with inverse given by the Lagrange
reversion formula: the series
\[
H(x,y) \defeq y + 
\sum_{k\ge1} \frac{1}{k!} \pa_y^{k-1}(R^k)(x,y)
\]
is formally convergent (the order of $\pa_y^{k-1}(R^k)$ is at least $k+1$
because the order of~$R$ is at least~$2$) and
satisfies $\th\big( x, H(x,y) \big) = y$.
Rewriting~\eqref{eqimplicitph} as $\th\big( x,\ph(x) \big) = f(x)$, we get
$\ph(x) = H\big( x,f(x) \big)$.

Now, the $y$-expansion of~$H$ can be written
$H(x,y) = \sum_{m\ge1} H_m(x) y^m$ with
\[
H_1 = (1-R_1)\ii  \quad\text{and}\quad
H_m = \sum_{k\ge1} \frac{(m+k-1)!}{m!\,k!} \sum_{\bn} R_{n_1} \cdots R_{n_k}
\ens\text{for $m\ge2$,}
\]
where the last summation is over all $k$-tuples of integers $\bn = (n_1,\ldots,n_k)$ such that
$n_1,\ldots,n_k \ge 1$ and $n_1 + \cdots + n_k = m+k-1$.
For $m\ge2$, grouping together the indices~$i$ such that $n_i = 1$, we get an expression
of~$H_m$ as a formally convergent series in $\C[[x]]$:
\beglabel{eqdefHm}
H_m = \sum_{r\ge0} \; \sum_{s\ge1}
\frac{(m+r+s-1)!}{m!\,r! \, s!} \sum_{\bj} R_1^r R_{j_1} \cdots R_{j_s},
\elabel
where the last summation is over all $s$-tuples of integers $\bj = (j_1,\ldots,j_s)$ such that
$j_1,\ldots,j_s \ge 2$ and $j_1 + \cdots + j_s = m+s-1$.
Observe that one must restrict oneself to $s \le m-1$ and that there
are $\binom{m-2}{s-1} \le 2^{m-2}$ summands in the $\bj$-summation.

Replacing $x$ by~$z\ii$, we get
\[
\ti\ph(z) = \ti H\big( z,\ti f(z) \big)
\]
with
$\ti f(z) \defeq f(z\ii) \in \ti\gR_\Om$ without constant term and
\[ \ti H(z,y) \defeq \sum_{m\ge1} \ti H_m(z) y^n, \qquad
\ti H_m(z) \defeq H_m(z\ii) \ens\text{for $m\ge1$.} \]
In view of Theorem~\ref{thmSubst} it is thus sufficient to check that $\ti H \in \ti\gR_\Om\{y\}$.

Let $K\subset\gS_\Om$ be compact.
Setting $\ti R_n(z) \defeq R_n(z\ii)$ for all $n\ge1$,
by Theorem~\ref{thmFrechetAlg} we can find $K'\supset K$ compact and $C>0$ such
that
$\norm*{\ti R_1^r \ti R_{j_1} \cdots \ti R_{j_s}}_K \le \frac{C^{r+s}}{r!}
\norm*{\ti R_1}_{K'}^r \norm*{\ti R_{j_1}}_{K'} \cdots \norm*{\ti R_{j_s}}_{K'}$.
Assuming $\ti F(z,y) \in \ti\gR_\Om\{y\}$, we can find $A,B>0$ such that
$\norm*{\ti R_n}_{K'} \le A B^n$ for all $n\ge1$.
Enlarging~$A$ if necessary, we can assume $3ABC \ge 1$.
We then see that the series~\eqref{eqdefHm} is convergent
in~$\ti\gR_\Om$:
for $m\ge2$,
\begin{align*}
\norm*{\ti H_m}_K & \le \sum_{r\ge0} \; \sum_{s=1}^{m-1}
\frac{(m+r+s-1)!}{m!\,r! \, s!} 
\frac{2^{m-2}C^{r+s}}{r!} A^{r+s} B^{m+r+s-1} \\[1ex]
& \le \frac{2^{m-2}}{m} \sum_{r\ge0} \frac{1}{r!} \sum_{s=1}^{m-1}
3^{m+r+s-1} (CA)^{r+s} B^{m+r+s-1} 
\le \dem (6B)^{m-1} \sum_{r\ge0} \frac{1}{r!} (3ABC)^{r+m-1},
\end{align*}
which is $\le \al \be^{m-1}$ with $\al = \dem\exp(3ABC)$ and $\be = 18 A B^2 C$. 
On the other hand, $\ti H_1 \in \ti\gR_\Om$ by Theorem~\ref{thmSubst}.
\end{proof}

%%%%%%%%%%%%%%%%%%%%%%%%%%%%%%%%%%%%%%%%%%%%%%%%%%%%%
%%%%%%%%%%%%%%%%%%%%%%%%%%%%%%%%%%%%%%%%%%%%%%%%%%%%%

\subsection{The group of resurgent tangent-to-identity diffeomorphisms}
\label{secgroupresurdiffeos}

%%%%%%%%%%%%%%%%%%%%%%%%%%%%%%%%%%%%%%%%%%%%%%%%%%%%%
%%%%%%%%%%%%%%%%%%%%%%%%%%%%%%%%%%%%%%%%%%%%%%%%%%%%%

One of the first applications by J.~\'Ecalle of his resurgence theory was
the iteration theory for tangent-to-identity local analytic diffeomorphisms
\cite[Vol.~2]{Eca81}.
In the language of holomorphic dynamics, this corresponds to a parabolic fixed
point in one complex variable, for which, classically, one introduces the Fatou
coordinates to describe the dynamics and to define the ``horn map''
\cite{Milnor}.
In the resurgent approach, one places the variable at infinity and deals with
formal diffeomorphisms:
starting from $F(w) = w + O(w^2) \in \C\{w\}$ or $\C[[w]]$, one gets 
$f(z) \defeq 1/F(1/z) = z + \sum_{m=0}^\infty a_m z^{-m} \in z + \C\{z\ii\}$ or
$z + \C[[z\ii]]$.
The set
\[
\ti\gG \defeq z + \C[[z\ii]]
\]
is a group for the composition law:
this is the group of formal tangent-to-identity diffeomorphisms.

Convergent diffeomorphisms form a subgroup $z + \C\{z\ii\}$.
In the simplest case, one is given a specific dynamical system 
$z \mapsto f(z) = z + \al + O(z\ii) \in z + \C\{z\ii\}$ with $\al\in\C^*$
and there is a formal conjugacy between~$f$ and the trivial dynamics $z \mapsto z+\al$, 
\ie the equation $\ti v\circ f = \ti v + \al$ admits a solution $\ti v\in \ti\gG$
(strictly speaking, an assumption is needed for this to be true, without which
one must enlarge slightly the theory to accept a logarithmic term in~$\ti v(z)$; we
omit the details here---see \cite{Eca81}, \cite{kokyu}).
One can give a direct proof \cite{EVinva} that $\ti v(z)-z$ is $\Om$-resurgent with $\Om =
{2\pi\I}{\al\ii}\Z$.
The inverse of~$\ti v$ is a solution~$\ti u$ of the difference equation $\ti u(z+\al) =
f\big(\ti u(z)\big)$
and the exponential of~$\ti v$ plays a role in \'Ecalle's ``bridge equation'' \cite{EVinvb}, which is
related to the \'Ecalle-Voronin classification theorem and to the horn map
(again, we refrain from giving more details here).

This may serve as a motivation for the following 
%
%%%%%%%%%%%%%%%%%%%%%%%%%%%%%%%%%%%%%%%%%%%%%%%%%%%%%

\begin{thm}	\label{thmGroup}
Assume that $\Om$ is a closed discrete subset of~$\C$ which contains~$0$ and is
stable under addition.
Then the $\Om$-resurgent tangent-to-identity diffeomorphisms
make up a subgroup 
\[
\ti\gG_\Om \defeq z + \ti\gR_\Om \subset \ti\gG,
\]
which contains $z+\C\{z\ii\}$.
\end{thm}

%%%%%%%%%%%%%%%%%%%%%%%%%%%%%%%%%%%%%%%%%%%%%%%%%%%%%

\begin{proof}
We must prove that, for arbitrary $\ti f(z) = z + \ti\phi(z), 
\ti g(z) = z + \ti\psi(z)\in \ti\gG_\Om$,
both $\ti f\circ\ti g$ and $\ti h \defeq \ti f\ic$ belong to~$\ti\gR_\Om$.

We have $\ti f\circ\ti g = \ti g + \ti\phi\circ\ti g$, where the last term can
be defined by the formally convergent series
\beglabel{eqdefcompos}
\ti\phi\circ\ti g = \ti\phi + \sum_{n\ge1}\frac{1}{n!}
\ti\psi^n \Big(\frac{\dd\,}{\dd z}\Big)^n \ti\phi.
\elabel
Let $K\subset \gS_\Om$ be compact, and let $K'\supset K$ and $C>0$ be as in
Theorem~\ref{thmFrechetAlg}. 
We have
\[
\norm*{\ti\psi^n \Big(\frac{\dd\,}{\dd z}\Big)^n \ti\phi}_K \le
C^{n+1} \norm*{\ti\psi}_{K'}^n \,
\norm*{\Big(\frac{\dd\,}{\dd z}\Big)^n \ti\phi}_{K'}
\le
C^{n+1} D(K')^n \norm*{\ti\psi}_{K'}^n \, \norm*{\ti\phi}_{K'},
\]
where $D(K') \defeq \max_{\ze\in K'}\abs{\ze}$ (by Remark~\ref{remactionddzcB}),
hence the series~\eqref{eqdefcompos} is convergent in~$\ti\gR_\Om$, and
$\norm*{\ti\phi\circ\ti g}_K \le C \norm*{\ti\phi}_{K'}
\exp\big( C D(K') \norm*{\ti\psi}_{K'} \big)$.

As for $\ti h$, the Lagrange reversion formula yields it in the form of a formally convergent series
\beglabel{eqdeftihinverstif}
\ti h = z + \sum_{k=1}^\infty \frac{(-1)^k}{k!} 
\Big(\frac{\dd\,}{\dd z}\Big)^{k-1}(\ti\phi^k).
\elabel
We have
\[
\norm*{\Big(\frac{\dd\,}{\dd z}\Big)^{k-1}(\ti\phi^k)}_K \le
D(K)^{k-1} \norm*{\ti\phi^k}_K \le
D(K)^{k-1} C^k \norm*{\ti\phi}_{K'}^k
\]
(again by Remark~\ref{remactionddzcB} and Theorem~\ref{thmFrechetAlg}),
hence the series~\eqref{eqdeftihinverstif} is convergent in~$\ti\gR_\Om$, and
$\norm*{\ti h-z}_K \le C \norm*{\ti\phi}_{K'}
\exp\big( C D(K) \norm*{\ti\phi}_{K'} \big)$.
\end{proof}

%%%%%%%%%%%%%%%%%%%%%%%%%%%%%%%%%%%%%%%%%%%%%%%%%%%%%

\begin{rem}
One can easily deduce from the estimates obtained in the above proof that
$\ti\gG_\Om$ is a topological group: composition and inversion are continuous if
we transport the topology of~$\ti\gR_\Om$ onto~$\ti\gG_\Om$ by the bijection
$\ti\phi \mapsto z + \ti\phi$.
\end{rem}

%%%%%%%%%%%%%%%%%%%%%%%%%%%%%%%%%%%%%%%%%%%%%%%%%%%%%
%%%%%%%%%%%%%%%%%%%%%%%%%%%%%%%%%%%%%%%%%%%%%%%%%%%%%

\subsection{Other applications}	\label{secPossAppl}

%%%%%%%%%%%%%%%%%%%%%%%%%%%%%%%%%%%%%%%%%%%%%%%%%%%%%
%%%%%%%%%%%%%%%%%%%%%%%%%%%%%%%%%%%%%%%%%%%%%%%%%%%%%

In this article, we stick to the simplest case which presents itself in
resurgence theory: 
formal expansions in negative integer powers of~$z$, whose Borel transforms
converge and extend analytically outside a set~$\Om$ fixed in advance, but 
\begin{enumerate}[--]
\item
the condition of $\Om$-continuability can be substituted with ``continuability
without a cut'' or ``endless continuability'' which allow for Riemann surfaces
much more general than~$\gS_\Om$ \cite[Vol.~3]{Eca81}, \cite{CNP};
\item
the theory of ``resurgent singularities'' was developed by J.~\'Ecalle to deal
with much more general formal objects than power series.
\end{enumerate}

%%%%%%%%%%%%%%%%%%%%%%%%%%%%%%%%%%%%%%%%%%%%%%%%%%%%%

The extension to more general Rieman surfaces is necessary in certain problems,
particularly those involving parametric resurgence or quantum resurgence (in
relation with semi-classical asymptotics).
To make our method accomodate the notion of continuability without a cut, one
could for instance imitate the way \cite{thesis_Y_Ou} deals with ``discrete
filtered sets''.
The point is that, when convolving germs in the $\ze$-plane, the singular points
of the analytic continuation of each factor may produce a singularity located at
the sum of these singular points,
but being continuable without a cut means that the set of singular points is
locally finite, thus one can explore sequentially the Riemann surface of the
convolution product, considering longer and longer paths of analytic
continuation and saturating the corresponding Riemann surface by removing at
each step the (finitely many) sums of singular points already encountered.

%%%%%%%%%%%%%%%%%%%%%%%%%%%%%%%%%%%%%%%%%%%%%%%%%%%%%

The formalism of general resurgent singularities also can be accomodated. 
The reader is referred to \cite{Eca81} and \cite{kokyu} for the corresponding
extension of the definition of convolution (see also \cite{EVinvb} and
\cite{introsummaresur}).
In short, the formal Borel transform~\eqref{eqdefcB}, which must be considered
as a termwise inverse Laplace transform, can be generalized by considering the
action of the Laplace transform on monomials like $\ze^\al (\log\ze)^m$ with
$m\in\N$ and $\al\in\C$ for instance.
One is thus led to deal with holomorphic functions of~$\ze$ defined for
arbitrarily small nonzero $\abs{\ze}$ but not holomorphic at the origin: one
must rather work in subsets of the Riemann surface of the logarithm 
(without even assuming the existence of any kind of expansion for small~$\abs{\ze}$)
before considering their analytic continuation for large values of~$\abs{\ze}$.
If one restricts oneself to functions which are integrable at~$0$, like the
convergent expansions involving monomials $\ze^\al (\log\ze)^m$ with
$\RE\al>-1$, then formula~\eqref{eqdefconvoldeux} may still be used to define
the convolution.
To deal with general resurgent singularities, one must replace it with the
so-called convolution of majors.
This should be the subject of another article,
but we can already mention that it is in the context of resurgent
singularities that the alien operators $\De_\om$ associated with non-zero
complex numbers~$\om$ are defined in the most efficient way.

These operators can be proved to be derivations (they satisfy the
Leibniz rule \wrt\ the convolution law) independent between them  and
independent of the natural derivation $\frac{\dd\,}{\dd z}$ except
for the relations
$\big[ \De_\om, \frac{\dd\,}{\dd z} \big] = - \om \De_\om$
(this is why they were called ``alien derivatives'' by \'Ecalle).
They annihilate the convergent series (because $\De_\om$ measures the
singularity at~$\om$ of a combination of branches of the Borel
transform and the Borel transform of a convergent series has no
singularity at all) and a suitable adaptation of
Theorem~\ref{thmboundconvga} allows one to check the rules of ``alien
calculus'', \eg
\begin{align*}
\De_\om\big( \ti H(\ti\ph_1,\ldots,\ti\ph_r) \big) &=
(\De_\om\ti H)(\ti\ph_1,\ldots,\ti\ph_r) +
\sum_{j=1}^r 
(\De_\om \ti\ph_j) \cdot
\frac{\pa\ti H}{\pa w_j}(\ti\ph_1,\ldots,\ti\ph_r) \\[1ex]
\De_\om(\ti f\circ \ti g) &=
\ee^{-\om(\ti g-z)} \cdot (\De_\om\ti f)\circ\ti g +
\Big(  \frac{\dd\ti f}{\dd z} \circ \ti g \Big) \cdot \De_\om\ti g
\end{align*}
in the situations of Theorems~\ref{thmSubst} and~\ref{thmGroup}
(where $\De_\om\ti H$ is defined,
with the notation of Theorem~\ref{thmSubst}, as the formal series 
$\sum (\De_\om \ti H_{\bk})(z) \, w_1^{k_1} \cdots w_r^{k_r}$,
and $(\De_\om\ti H)(\ti\ph_1,\ldots,\ti\ph_r)$
and $(\De_\om\ti f)\circ\ti g$
must be defined properly; see Theorem~30.9 of \cite{introsummaresur}
for an example).

%%%%%%%%%%%%%%%%%%%%%%%%%%%%%%%%%%%%%%%%%%%%%%%%%%%%%

As another possible application, it would be worth trying to adapt our method to
the weighted convolution products which appear in \cite{Eca94}. 
Their definition is as follows:
given a sequence of pairs $B_1 = (\om_1,b_1)$, $B_2 = (\om_2,b_2)$, etc.\ 
with $\om_n \in \C$ and $b_n \in \C\{\ze\}$ and assuming that
\[
\check\om_n = \om_1 + \cdots + \om_n \neq 0, \qquad n\in\N^*,
\]
one defines a sequence $\hat S^{B_1}, \hat S^{B_1,B_2}, \ldots \in \C\{\ze\}$ by
the formulas
\begin{multline*}
\hat S^{B_1}(\ze) \defeq \frac{1}{\om_1} b_1\Big( \frac{\ze}{\om_1}\Big),
\qquad
\hat S^{B_1,B_2}(\ze) \defeq \frac{1}{\om_1} \int_0^{\ze/\check\om_2}
b_1\Big( \frac{\ze-\om_2 \xi_2}{\om_1} \Big) b_2(\xi_2) \,\dd\xi_2, 
\\[1ex]
\hat S^{B_1,B_2,B_3}(\ze) \defeq \frac{1}{\om_1} \int_0^{\ze/\check\om_3} \dd\xi_3
\int_{\xi_3}^{(\ze-\om_3\xi_3)/\check\om_2} \dd\xi_2 \,
b_1\Big( \frac{\ze-\om_2 \xi_2 - \om_3 \xi_3}{\om_1} \Big) 
b_2(\xi_2) b_3(\xi_3),
\qquad \text{etc.}
\end{multline*}
The general formula is
$\dst \hat S^{B_1,\ldots,B_n}(\ze) \defeq \frac{1}{\om_1} \dst\int
\dd\xi_n \cdots \dd\xi_2 \, b_1(\xi_1) b_2(\xi_2) \cdots b_n(\xi_n)$,
where the integral is taken over 
\[
\xi_n \in \Big[ 0, \frac{\ze}{\check\om_n} \Big],
\qquad
\xi_i \in \Big[ \xi_{i+1}, 
\frac{\ze - (\om_{i+1}\xi_{i+1} + \cdots + \om_n\xi_n)}{\check\om_i} \Big]
\ens \text{for $i = n-1, n-2, \ldots, 2$}
\]
and $\dst \xi_1 \defeq \frac{\ze - (\om_2\xi_2 + \cdots + \om_n\xi_n)}{\check\om_1}$.
There is a relation with the ordinary convolution called symmetrality: if
$\bB' = B^{i_1}\cdots B^{i_n}$ and
$\bB'' = B^{j_1}\cdots B^{j_m}$, then 
$\hat S^{\bB'} * \hat S^{\bB''}$ is the sum $\sum \hat S^{\bB}$ over all
words~$\bB$ belonging to the shuffle of~$\bB'$ and~$\bB''$, \eg
\[
\hat S^{B_1} * \hat S^{B_2} = \hat S^{B_1,B_2} + \hat S^{B_2,B_1},
\qquad
\hat S^{B_1,B_2} * \hat S^{B_3} = \hat S^{B_1,B_2,B_3} + \hat S^{B_1,B_3,B_2}
+ \hat S^{B_3,B_1,B_2},
\qquad \text{etc.}
\]
It is argued in \cite{Eca94} that the weighted convolutions $\hat S^{B_1,\ldots,B_n}$
associated with endlessly continuable germs $b_1, b_2, \ldots$ are themselves
endlessly continuable and constitute the ``building blocks'' of the resurgent
functions which appear in parametric resurgence or quantum resurgence problems
(see \cite{AnnF} for an example with $\om_i = 1$ for all~$i$).
It would thus be interesting and natural (because the weighted convolution
products present themselves as multiple integrals not so different from the
$n$-fold integrals~\eqref{eqdefconvmultint} below) to try to deform the
integration simplex, in a manner similar to the one that will be employed for
convolution products in Sections~\ref{secstartpf}--\ref{secestimisot}, in order to
control the analytic continuation of $\hat S^{B_1,\ldots,B_r}$.

%%%%%%%%%%%%%%%%%%%%%%%%%%%%%%%%%%%%%%%%%%%%%%%%%%%%%
%%%%%%%%%%%%%%%%%%%%%%%%%%%%%%%%%%%%%%%%%%%%%%%%%%%%%

\subsection{Nonlinear analysis with $1$-summable series}	\label{secNLsumma}

%%%%%%%%%%%%%%%%%%%%%%%%%%%%%%%%%%%%%%%%%%%%%%%%%%%%%
%%%%%%%%%%%%%%%%%%%%%%%%%%%%%%%%%%%%%%%%%%%%%%%%%%%%%

For the resurgent series encountered in practice, one is often
interested in applying Borel-Laplace summation. It is thus important
to notice that the property of $1$-summability too is compatible with
the nonlinear operations described in the previous sections.

We recall that, given a non-trivial interval~$\cA$ of~$\R$, a formal series
$\ti\phi(z) \in \C[[z\ii]]$ is said to be \emph{$1$-summable in the
  directions of~$\cA$} if 
it can be written $\ti\phi = c + \cB\ii\hat\ph$ with $c\in\C$ and
$\hat\ph(\ze) \in \C\{\ze\}$,
there exists $\rho>0$ such that~$\hat\ph$ extends analytically to
\[
S(\rho,\cA) \defeq \D_\rho \cup \big\{ r\,\eith \mid r>0,\;
\th\in\cA \big\}
\]
and there exist $\tau \in \R$ and $C>0$ such that 
$\abs{\hat\ph(\ze)} \le C\, \ee^{\tau\abs{\ze}}$ for all $\ze \in S(\rho,\cA)$.
In such a case, the \emph{Borel sum of~$\ti\phi$} is the function
obtained by glueing the Laplace transforms of~$\hat\ph$ associated with
the directions of~$\cA$ (the Cauchy Theorem entails that they match)
and adding the constant term~$c$,
\ie the function $\SUM^{\cA}\ti\phi$ holomorphic in the union of half-planes%
\footnote{viewed as a subset of the Riemann surface of the logarithm if
$\tau\ge0$ and $\cA$ has length $\ge\pi$}
$\Sig^\cA_\tau \defeq \dst \bigcup_{\th\in\cA} \{ z \mid \RE(z\,\eith) > \tau \}$
defined by
$\dst
\big( \SUM^{\cA}\ti\phi \big)(z) \defeq c + \int_0^{\eith\infty} \hat\ph(\ze) \,\ee^{-z\ze}\,\dd\ze$
% \quad \text{
%
with any $\th\in\cA$ such that $\RE(z\,\eith) > \tau$.
This function admits~$\ti\phi(z)$ as Gevrey asymptotic expansion and
is the only one with this property---see \eg \cite{Bal94}, \cite{gazRamis}.

%%%%%%%%%%%%%%%%%%%%%%%%%%%%%%%%%%%%%%%%%%%%%%%%%%%%%

Let us denote by $\ti\gS^\cA$ the subspace of all $1$-summable
series, so that
\[\C\{z\ii\} \subset \ti\gS^\cA \subset \C[[z\ii]] \]
($\SUM^\cA$ coincides with ordinary summation in restriction to $\C\{z\ii\}$).
%
%%%%%%%%%%%%%%%%%%%%%%%%%%%%%%%%%%%%%%%%%%%%%%%%%%%%%
%
The open sets $S(\rho,\cA)$ being star-shaped \wrt~$0$, we can
use~\eqref{eqiterconvol} and check that the properties imposed to the
$\hat\ph$'s to define $1$-summability (analytic continuation to
$S(\rho,\cA)$ and exponential bound) are stable under convolution.
More precisely, we get
\begin{multline}
\text{$\abs{\hat\ph_j(\ze)} \le C_j \, \ee^{\tau\abs{\ze}}$ 
for $\ze \in S(\rho,\cA)$ and $j=1,\ldots,n$} \\[1ex]
\label{ineqexpbound}
\Longrightarrow\quad
\abs{ \hat\ph_1*\cdots*\hat\ph_n(\ze) } \le 
\frac{\abs{\ze}^{n-1}}{(n-1)!} 
C_1 \cdots C_n \, \ee^{\tau\abs{\ze}}
\ens\text{for $\ze \in S(\rho,\cA)$.}
\end{multline}
It follows that $\ti\gS^\cA$ is a subalgebra of $\C[[z\ii]]$ and, since the Laplace transform maps
the convolution product onto the multiplication of functions,
$\SUM^{\cA}(\ti\phi_1 \ti\phi_2) = ( \SUM^{\cA}\ti\phi_1 ) (
\SUM^{\cA}\ti\phi_2 )$.
%
%%%%%%%%%%%%%%%%%%%%%%%%%%%%%%%%%%%%%%%%%%%%%%%%%%%%%
%
In view of Remark~\ref{remactionddzcB}, it is even a differential
subalgebra and $\SUM^{\cA}\frac{\dd\ti\phi}{\dd z} = 
\frac{\dd\;}{\dd z} \SUM^{\cA}\ti\phi$.

%%%%%%%%%%%%%%%%%%%%%%%%%%%%%%%%%%%%%%%%%%%%%%%%%%%%%

To go farther, we fix a non-trivial interval~$\cA$ of~$\R$ and set
\[
\norm*{\ti\phi}_{\rho,\tau} \defeq 
\abs{c} + \sup_{S(\rho,\cA)} \ee^{-\tau\abs{\ze}} \abs{\hat\ph(\ze)}
\]
for any $\rho>0$, $\tau \in \R$ and
$\ti\phi = c + \cB\ii\hat\ph \in \C \oplus \cB\ii\big(
\gO(S(\rho,\cA)) \big)$,
so that a formal series~$\ti\phi(z)$ belongs to~$\ti\gS^\cA$ if and
only if there exist~$\rho$ and~$\tau$ such that
$\norm*{\ti\phi}_{\rho,\tau} < \infty$,
and $\SUM^\cA\ti\phi$ is then holomorphic at least in~$\Sig^\cA_\tau$.
The results of the previous sections can be complemented with the
following four theorems, the proof of which will be outlined at the
end of this section:

%%%%%%%%%%%%%%%%%%%%%%%%%%%%%%%%%%%%%%%%%%%%%%%%%%%%%
%
\begin{thmPdt}
Suppose $n\ge1$, $\ti\phi_1,\ldots,\ti\phi_n \in \ti\gS^\cA$ and $N\in\N$. Then
\[
\norm*{\ti\phi_1 \cdots \ti\phi_n}_{\rho,\tau+\eps} \le
\max(1,\tfrac{1}{\eps^{n-1}})
\norm*{\ti\phi_1}_{\rho,\tau} \cdots \norm*{\ti\phi_n}_{\rho,\tau},
\qquad
\norm{ \vphantom{\ti\phi} \smash{\frac{\dd^N\ti\phi_1}{\dd z^N}}}_{\rho,\tau+\eps} \le
\frac{N!}{\eps^N} \norm*{\ti\phi_1}_{\rho,\tau}
\]
for every $\rho>0$, $\tau\in\R$ and $\eps>0$.
\end{thmPdt}

%%%%%%%%%%%%%%%%%%%%%%%%%%%%%%%%%%%%%%%%%%%%%%%%%%%%%
%
\begin{thmSubs}
Suppose that $\ti H
= \sum_{\bk = (k_1,\ldots,k_r) \in \N^r} \ti H_{\bk}(z) \, w_1^{k_1} \cdots w_r^{k_r}
\in \C[[z\ii,w_1,\ldots,w_r]]$ has its coefficients $1$-summable in
the directions of~$\cA$ and
$\norm*{\ti H_{\bk}}_{\rho,\tau} \le A \, B^{\abs{\bk}}$ for all $\bk\in\N^r$,
with some $\rho,A,B>0$ and $\tau\in\R$ independent of~$\bk$.
Then 
\[
\cH^\cA(z,w_1,\ldots,w_r) \defeq 
\sum_{\bk = (k_1,\ldots,k_r) \in \N^r} 
(\SUM^\cA\ti H_{\bk})(z) \, w_1^{k_1} \cdots w_r^{k_r}
\]
is holomorphic in $\Sig^\cA_\tau \times \D_{1/B} \cdots \times \D_{1/B}$
and, for all $\ti\ph_1,\ldots,\ti\ph_r \in \ti\gS^\cA$ without
constant term, 
\[
\ti H(\ti\ph_1,\ldots,\ti\ph_r) \in \ti\gS^\cA 
\ens \text{and} \ens
\SUM^\cA \big( \ti H(\ti\ph_1,\ldots,\ti\ph_r) \big)(z) =
\cH^\cA \big( z, \SUM^\cA\ti\ph_1(z), \ldots, \SUM^\cA\ti\ph_r(z) \big)
\]
for $z \in \Sig^\cA_{\tau'}$
as soon as~$\tau'$ is large enough.
One can take
$\tau' = \tau + B\big(
\norm*{\ti\ph_1}_{\rho,\tau} + \cdots + \norm*{\ti\ph_r}_{\rho,\tau}
\big)$
if this number is finite
(if not, take~$\tau$ larger and~$\rho$ smaller), in which case
\[
\norm*{\ti H(\ti\ph_1,\ldots,\ti\ph_r)}_{\rho,\tau'}
\le A\Big( 2 + B\big(
\norm*{\ti\ph_1}_{\rho,\tau} + \cdots + \norm*{\ti\ph_r}_{\rho,\tau}
\big) \Big).
\]
\end{thmSubs}

%%%%%%%%%%%%%%%%%%%%%%%%%%%%%%%%%%%%%%%%%%%%%%%%%%%%%
%
\begin{thmTFI}
Suppose that $\ti F = \sum_{k\ge0} \ti F_k(z) y^k \in \C[[z\ii,y]]$
and $\norm*{\ti F_k}_{\rho,\tau} \le A B^k$ for all $k\in\N$,
with some $\rho,A,B>0$ and $\tau\in\R$ independent of~$k$,
so that
\[
\cF^\cA(z,y) \defeq \sum_{k\ge0} (\SUM^\cA\ti F_k)(z) \, y^k
\]
is holomorphic in $\Sig^\cA_\tau \times \D_{1/B}$.
Suppose moreover that $\ti F_0(z) \in \zcz$ and that the constant term
of~$\ti F_1$ is nonzero, so that the equation 
$\ti F\big(z,\ti\ph(z)\big) = 0$
implicitly defines a formal series $\ti\ph(z) \in \zcz$.
Then this unique formal solution is $1$-summable in the directions
of~$\cA$,
and $\SUM^\cA\ti\ph$ is a solution of the corresponding
functional equation
$\cF^\cA\big(z,(\SUM^\cA\ti\ph)(z)\big) = 0$
which is holomorphic in~$\Sig^\cA_{\tau'}$
for $\tau'$ large enough.
\end{thmTFI}

%%%%%%%%%%%%%%%%%%%%%%%%%%%%%%%%%%%%%%%%%%%%%%%%%%%%%

We also define the set of ``$1$-summable tangent-to-identity diffeomorphisms''
\[ \ti\gG^\cA \defeq z+\ti\gS^\cA \subset \ti\gG = z + \C[[z\ii]] \]
and use the notations
$\norm*{\ti f}_{\rho,\tau} \defeq \norm*{\ti\phi}_{\rho,\tau}$
and $\SUM^\cA\ti f \defeq \ID+\SUM^\cA\ti\phi$
for any $\ti f(z) = z + \ti\phi(z) \in \ti\gG^\cA$.

\begin{thmGrp}
The set $\ti\gG^\cA$ is a subgroup of~$\ti\gG$ and contains $z+\C\{z\ii\}$.
For $\ti f,\ti g \in \ti\gG^\cA$ with 
$\norm*{\ti f}_{\rho,\tau}, \norm*{\ti g}_{\rho,\tau} < \infty$,
one has
\[
\norm*{\ti g \circ \ti f}_{\rho,\tau'} \le 
\norm*{\ti f}_{\rho,\tau}  + \norm*{\ti g}_{\rho,\tau},
\qquad
\norm*{\ti f\ic}_{\rho,\tau'} \le \norm*{\ti f}_{\rho,\tau} ,
\]
with $\tau' \defeq \tau + 1 + \norm*{\ti f}_{\rho,\tau}$.
Moreover, the composition $\big(\SUM^\cA\ti g\big) \circ \big(\SUM^\cA\ti f\big)$ is
well-defined and coincides with $\SUM^\cA\big( \ti g \circ \ti f \big)$ on~$\Sig^\cA_{\tau'}$ 
and, for $\tau''$ large enough, $\SUM^\cA\ti f$ is injective on $\Sig^\cA_{\tau''}$
and the composition $\big(\SUM^\cA(\ti f\ic)\big) \circ \big(\SUM^\cA\ti f\big)$ is well-defined and
coincides with~$\ID$ on $\Sig^\cA_{\tau''}$.
\end{thmGrp}

%%%%%%%%%%%%%%%%%%%%%%%%%%%%%%%%%%%%%%%%%%%%%%%%%%%%%

Before proceeding with the proof of these statements, we first mention

\begin{lem}   \label{lemestimSUM}
Let $\ti\phi \in \ti\gS^\cA$ with $\norm*{\ti\phi}_{\rho,\tau} <
\infty$. Then
\[
\abs*{\SUM^\cA\ti\phi(z)} \le D(z) \norm*{\ti\phi}_{\rho,\tau}
\quad \text{for $z\in \Sig^\cA_\tau$,}
\quad\text{with}\ens
D(z) = \max \Big( 1, \frac{1}{\sup_{\th\in\cA}\RE(z\,\eith-\tau)} \Big).
\]
If $\ti\phi$ has no constant term and $z\in \Sig^\cA_{\tau'}$ with
$\tau'>\tau$, then one can take
$D(z) = \frac{1}{\tau'-\tau}$.
\end{lem}

%%%%%%%%%%%%%%%%%%%%%%%%%%%%%%%%%%%%%%%%%%%%%%%%%%%%%

\begin{proof}
Write $\ti\phi = c + \ti\ph$ with $\ti\ph$ without constant term and
take $z \in \Sig^\cA_\tau$. For any $\th \in \cA$ such that 
$\de_\th(z) \defeq \RE(z\,\eith-\tau)>0$, we have
$\abs*{\SUM^\cA\ti\ph(z)} \le \frac{1}{\de_\th(z)} \norm*{\ti\ph}_{\rho,\tau}$,
whence the conclusion follows.
\end{proof}

%%%%%%%%%%%%%%%%%%%%%%%%%%%%%%%%%%%%%%%%%%%%%%%%%%%%%

\begin{proof}[Outline of the proof of Theorem~\ref{thmFrechetAlg}']
For $j=1,\ldots,n$, we write $\ti\phi_j = c_j + \ti\ph_j$ with
$c_j\in\C$ and $\ti\ph_j$ without constant term.
For $N\ge1$, $\frac{\dd^N\ti\phi_1}{\dd z^N}$ is the inverse Borel
transform of $(-\ze)^N \hat\ph_1(\ze)$, whose modulus is bounded on
$S(\rho,\cA)$ by
$\frac{N!}{\eps^N} \frac{ (\eps\abs{\ze})^N }{N!}
\norm*{\ti\ph_1}_{\rho,\tau} \, \ee^{\tau\abs{\ze}} \le
\frac{N!}{\eps^N} \norm*{\ti\phi_1}_{\rho,\tau} \, \ee^{ (\tau+\eps) \abs{\ze}}$.
%
% which yields the second statement.

The first statement results from the identity 
$\ti\phi_1 \cdots \ti\phi_n = c + \ti\psi$ with $c = c_1 \cdots c_n$
and % $\ti\psi$ inverse Borel transform of
$\hat\psi = \sum
c_{i_1} \cdots c_{i_p} 
\hat\ph_{j_1} * \cdots * \hat\ph_{j_q}$,
with summation over all proper subsets 
$I = \{ i_1, \ldots, i_p \}$ of $\{1,\ldots,n\}$ 
%
% (of any cardinality $p<n$) 
%
and $\{j_1,\ldots,j_q\} \defeq \{1,\ldots,s\}\setminus I$:
\eqref{ineqexpbound} yields
$\abs*{\hat\psi(\ze)} \le \sum \abs*{c_{i_1} \cdots c_{i_p}}
\norm*{\ti\ph_{j_1}}_{\rho,\tau} \cdots \norm*{\ti\ph_{j_q}}_{\rho,\tau}
\frac{\abs{\ze}^{q-1}}{(q-1)!} \ee^{\tau\abs{\ze}}$ on $S(\rho,\cA)$
and 
$\frac{\abs{\ze}^{q-1}}{(q-1)!} \le \max(1,\tfrac{1}{\eps^{n-1}}) \ee^{\eps\abs{\ze}}$,
%
% hence $\abs*{\hat\psi(\ze)} \le \max(1,\tfrac{1}{\eps^{n-1}}) 
%
% \sum \abs*{c_{i_1} \cdots c_{i_p}}
%
% \norm*{\ti\ph_{j_1}}_{\rho,\tau} \cdots \norm*{\ti\ph_{j_q}}_{\rho,\tau}
%
% \ee^{(\tau+\eps)\abs{\ze}}$,
%
while % on the other hand 
$\abs{c} \le \max(1,\tfrac{1}{\eps^{n-1}})
\abs*{c_{i_1} \cdots c_{i_p}}$.
\end{proof}

%%%%%%%%%%%%%%%%%%%%%%%%%%%%%%%%%%%%%%%%%%%%%%%%%%%%%

\begin{rem}   \label{remVarSubsP}
A simple modification of the previous argument, in the spirit of the
proof of Theorem~\ref{thmFrechetAlg}, shows that, if among
$\ti\phi_1,\ldots,\ti\phi_n$ at least $r\ge1$ formal series have no
constant term, then $\ti\phi_1 \cdots \ti\phi_n = \ti\psi$ is the
inverse Borel transform of a function which satisfies
\[
\abs*{\hat\psi(\ze)} \le \max(1,\tfrac{1}{\eps^{n-r}}) 
\norm*{\ti\phi_1}_{\rho,\tau} \cdots \norm*{\ti\phi_n}_{\rho,\tau}
\frac{\abs{\ze}^{r-1}}{(r-1)!} \ee^{(\tau+\eps)\abs{\ze}}
\quad\text{for $\ze \in S(\rho,\cA)$.}
\]
\end{rem}

%%%%%%%%%%%%%%%%%%%%%%%%%%%%%%%%%%%%%%%%%%%%%%%%%%%%%

\begin{proof}[Outline of the proof of Theorem~\ref{thmSubst}']
The first statement follows from Lemma~\ref{lemestimSUM}.
Suppose $C_j \defeq \norm*{\ti\ph_j}_{\rho,\tau} < \infty$ for each~$j$,
then Lemma~\ref{lemestimSUM} also shows that 
$\cH^\cA \big( z, \SUM^\cA\ti\ph_1(z), \ldots, \SUM^\cA\ti\ph_r(z)
\big)$ is well-defined for $z\in\Sig^\cA_{\tau'}$ as soon as
$\tau'-\tau > B \max(C_1,\ldots,C_n)$.

Let us write $\ti H_k = c_k + \ti G_k$ with $\ti G_k$ without constant term.
Then $\ti H(\ti\ph_1,\ldots,\ti\ph_r) = c_0 + \ti\psi$,
\[
\ti\psi = \cB\ii\hat\psi, \qquad
\hat\psi = \sum_{\bk\in\N^r\setminus\{0\}} c_{\bk} \hat\ph_1^{k_1} * \cdots * \hat\ph_r^{k_r}
+ \sum_{\bk\in\N^r} \hat G_{\bk} * \hat\ph_1^{k_1} * \cdots * \hat\ph_r^{k_r}.
\]
By inequality~\eqref{ineqexpbound}, 
representing by $(\bbe_1,\ldots,\bbe_r)$ the canonical basis of~$\R^r$,
we get
$\abs*{\hat\psi(\ze)} \le
\sum_{\bk\in\N^r} \big( 
\sum_{j=1}^r \abs*{c_{\bk+\bbe_j}} C_j + \norm*{\ti G_{\bk}}_{\rho,\tau}
\big) C_1^{k_1} \cdots C_r^{k_r} \frac{\abs{\ze}^{\abs\bk}}{\abs\bk!} \ee^{\tau\abs{\ze}}
\le A \big( B(C_1+\cdots+C_r) + 1 \big) \ee^{\tau'\abs{\ze}}$.
%
% where $\tau' \defeq \tau + B(C_1+\cdots+C_r)$.
%
\end{proof}

%%%%%%%%%%%%%%%%%%%%%%%%%%%%%%%%%%%%%%%%%%%%%%%%%%%%%

\begin{proof}[Outline of the proof of Theorem~\ref{thmImplicit}']
Dividing ~$\ti F$ by the appropriate factor, we can suppose that
$\ti F(z,y) = - y + \ti f(z) + \sum_{n\ge1} \ti R_n(z) y^n$,
where $\ti f$ and~$\ti R_1$ have no constant term, 
$\norm*{\ti   f}_{\rho,\tau} \le A$ and $\norm*{\ti R_n}_{\rho,\tau} \le A B^n$ for every
$n\ge 1$.
Arguing as in the proof of Theorem~\ref{thmImplicit}, we have
$\ti\ph = \sum_{m\ge1} \ti H_m \ti f^m$
with $\ti H_1 \defeq (1-\ti R_1)\ii$
and $\ti H_m$ given by~\eqref{eqdefHm} for $m\ge2$.
The main task consists in showing that there exist $\al,\be>0$
and~$\tau_1$ such that
$\norm*{\ti H_m}_{\rho,\tau_1} \le \al \be^m$ for all~$m$,
so that Theorem~\ref{thmSubst}' can be applied.

In the $r$-summation defining~$\ti H_m$, we separate $r=0$ and
$r\ge1$, so that $\ti H_m = \ti H'_m + \ti H''_m$ with
$\ti H'_1 \defeq 1$, $\ti H''_1 \defeq \sum_{r\ge1} \ti R_1^r$, and,
for $m\ge2$,
\[
\ti H'_m \defeq \sum_{s=1}^{m-1} \frac{(m+s-1)!}{m!\,s!} 
\sum_{\bj} \ti R_{j_1} \cdots \ti R_{j_s}, 
\quad
\ti H''_m \defeq \sum_{r\ge1} \sum_{s=1}^{m-1} 
\frac{(m+r+s-1)!}{m!\,r! \, s!} \sum_{\bj} \ti R_1^r \ti R_{j_1} \cdots \ti R_{j_s}
\]
with summation over all $\bj = (j_1,\ldots,j_s)$ such that
$j_1,\ldots,j_s \ge 2$ and $j_1 + \cdots + j_s = m+s-1$.
Enlarging~$A$ if necessary, we suppose $2AB\ge1$.
Theorem~\ref{thmSubst}' yields, for $m\ge2$,
$\norm*{\ti H'_m}_{\rho,\tau+1} \le$
\[
\sum_{s=1}^{m-1} \frac{(m+s-1)!}{m!\, s!} \binom{m-2}{s-1} A^s B^{m+s-1}
\le \sum_{s=1}^{m-1} \frac{2^{m+s-1}}{m} 2^{m-2} (AB)^s B^{m-1}
\le \frac{(4B)^{m-1}}{2m} \sum_{s=1}^{m-1} (2AB)^s,
\]
whence
$\norm*{\ti H'_m}_{\rho,\tau+1} \le (8AB^2)^{m-1}$ for all $m\ge1$.
On the other hand, $\norm*{\ti H''_1}_{\rho,\tau+AB} < \infty$ by Theorem~\ref{thmSubst}'
and, by Remark~\ref{remVarSubsP},
for $m\ge2$ and $\ze \in S(\rho,\cA)$, 
\begin{multline*}
\abs*{\hat H''_m(\ze)} \ee^{-(\tau+1)\abs{\ze}} 
\le \sum_{r\ge1} \sum_{s=1}^{m-1} 
\frac{(m+r+s-1)!}{m!\,r! \, s!} 
\binom{m-2}{s-1}
A^{r+s} B^{m+r+s-1} \frac{\abs{\ze}^{r-1}}{(r-1)!} \\
\le \sum_{r\ge1} \sum_{s=1}^{m-1} \frac{3^{m+r+s-1}}{m} 2^{m-2}
(AB)^{r+s} B^{m-1} \frac{\abs{\ze}^{r-1}}{(r-1)!} \\
\le \dem (6B)^{m-1} \sum_{r\ge1} (3AB)^{r+m-1} \frac{\abs{\ze}^{r-1}}{(r-1)!}
\end{multline*}
(because $3AB\ge1$), whence $\norm*{\ti H''_m}_{\rho,\tau+1+3AB} \le \frac{3AB}{2} (18 AB^2)^{m-1}$.
\end{proof}

%%%%%%%%%%%%%%%%%%%%%%%%%%%%%%%%%%%%%%%%%%%%%%%%%%%%%

\begin{proof}[Outline of the proof of Theorem~\ref{thmGroup}']
Write $\ti f = z + c + \ti\ph(z)$, $\ti g = z + c' + \ti\psi(z)$, with
$c,c' \in \C$,
$\ti\ph$ and~$\ti\psi$ without constant term,
and let $A \defeq \norm*{\ti\ph}_{\rho,\tau}$ and $B \defeq \norm*{\ti\psi}_{\rho,\tau}$.
The function $\big(\SUM^\cA\ti g\big) \circ \big(\SUM^\cA\ti f\big)$ is
well-defined on~$\Sig^\cA_{\tau'}$ because, by Lemma~\ref{lemestimSUM},
$\abs*{\SUM^\cA\ti f(z) - z} \le \abs{c} + \frac{A}{\tau'-\tau} \le \tau'-\tau$.

Let $\ti\ph_0(z) \defeq \ti\ph(z-c)$: 
we have $\ti g\circ\ti f = z + c + c' + \ti\ph + \ti\chi$ with
\[
\ti\chi \defeq \ti\chi_0 \circ (\ID+c), \ens \ti\chi_0 \defeq \ti\psi \circ (\ID+\ti\ph_0),
\quad \text{hence} \ens
\hat\chi(\ze)  = \hat\chi_0(\ze)\,\ee^{-c\ze}, \ens
\hat\chi_0(\ze) = \sum_{k\ge0} 
\big( \tfrac{(-\ze)^k}{k!} \hat\psi \big) * \hat\ph_0^{*k}.
\]
Given $\th\in\cA$ and $\tau_+ \defeq \tau + \RE(c\,\eith)$, 
since $\hat\ph_0(\ze) = \ee^{c\ze}\hat\ph(\ze)$, we have 
$\abs{\hat\ph_0(\ze)} \le A \, \ee^{\tau_+\abs{\ze}}$ on $\R^+\,\eith$,
thus
$\abs*{\big( \tfrac{(-\ze)^k}{k!} \hat\psi \big) * \hat\ph_0^{*k}(\ze)}
\le B A^k \frac{\abs{\ze}^{2k}}{(2k)!} \ee^{ \max\{\tau,\tau_+\} \abs{\ze}}$
and $\abs*{\hat\chi_0(\ze)} \le B \, \ee^{(\sqrt{A} + \max\{\tau,\tau_+\}) \abs{\ze}}$,
whence $\abs*{\hat\chi(\ze)} \le B \, \ee^{(\sqrt{A} + \tau+\abs{c}) \abs{\ze}}
\le B \, \ee^{\tau'\abs{\ze}}$
and $\norm*{\ti g \circ \ti f}_{\rho,\tau'} \le 
\abs{c}+\abs*{c'}+A+B$.

Using the Lagrange reversion formula to compute 
$\ti f\ic = (\ID+\ti\ph_0)\ii - c$, we get
\[
\ti f\ic = \ID - c + \ti\ph_-, \quad
\ti\ph_- \defeq \sum_{k\ge1} \tfrac{(-1)^k}{k!} 
\big( \tfrac{\dd\,\,}{\dd z} \big)^{k-1} \ti\ph_0^k,
\quad\text{hence}\ens
\hat\ph_-(\ze) = - \sum_{k\ge1} \tfrac{\ze^{k-1}}{k!} \hat\ph_0^{*k}.
\]
On $S(\rho,\cA)$, 
$\abs*{\hat\ph_0^{*k}(\ze)} \le A^k \frac{\abs{\ze}^{k-1}}{(k-1)!} \ee^{(\tau+\abs{c})\abs{\ze}}$,
thus $\abs{\hat\ph_-(\ze)} \le A \, \ee^{(\tau+2\sqrt{A})\abs{\ze}}
\le A\, \ee^{\tau'\abs{\ze}}$.
\end{proof}

%%%%%%%%%%%%%%%%%%%%%%%%%%%%%%%%%%%%%%%%%%%%%%%%%%%%%
%%%%%%%%%%%%%%%%%%%%%%%%%%%%%%%%%%%%%%%%%%%%%%%%%%%%%

\section{The initial $n$-dimensional integration current}	\label{secstartpf}

%%%%%%%%%%%%%%%%%%%%%%%%%%%%%%%%%%%%%%%%%%%%%%%%%%%%%
%%%%%%%%%%%%%%%%%%%%%%%%%%%%%%%%%%%%%%%%%%%%%%%%%%%%%

We now begin the proof of Theorem~\ref{thmboundconvga}.
Notice that convolution with the constant germ~$1$ amounts to integration
from~$0$, according to~\eqref{eqdefconvoldeux}, thus
$\frac{\dd\,\,}{\dd\ze}(1*\hat\ph) = \hat\ph$
and, by associativity of the convolution,
\beglabel{eqrelderiv} 
\hat\ph_1 * \cdots * \hat\ph_n = \frac{\dd\,\,}{\dd\ze}
\big(1 * \hat\ph_1 * \cdots * \hat\ph_n \big)
\elabel
for any $\hat\ph_1,\ldots,\hat\ph_n \in \C\{\ze\}$.

We shall now dedicate ourselves to the proof of a statement similar
to Theorem~\ref{thmboundconvga} for convolution products of the form
$1*\hat\ph_1*\cdots*\hat\ph_n$, with $\hat\ph_1,\ldots, \hat\ph_n \in \gO(\gS_\Om)$;
this will be Theorem~\ref{thmboundconvga}' of Section~\ref{secestimisot}.
The proof of Theorem~\ref{thmboundconvga} itself will then follow by the Cauchy
inequalities.

It turns out that, for $\ze\in\gS_\Om$ close to~$0_\Om$, there is a natural way
of representing $1*\hat\ph_1*\cdots*\hat\ph_n(\ze)$ as the integral of a
holomorphic $n$-form over an $n$-dimensional chain
%
% \footnote{%
% %
% See \eg \cite{Shabat} for a reminder about cells, chains, the boundary
% operator~$\pa$ and the integration of differential forms on chains.
% %
% As far as regularity is concerned, we shall require a slight extension of the
% classical theory to the case where cells are defined by Lipschitz maps
% %
% $C \col Q \to M$, with $M$ a complex manifold and $Q$ of the form $[0,1]^m\times
% \De_n$ (with Notation~\ref{notesimplexn}), which are $C^1$ on the
% interior~$\INT Q$ of~$Q$. The integral of a smooth differential form~$\al$
% on~$C$ is then defined as usual as the integral on~$Q$ of the pullback $C^*\al$
% (the coefficients of which are continuous and uniformy bounded on the interior
% of~$Q$). DONNER UNE REF (Federer, Whitney, Heinonen, Kuttler?)
% %
% }
%
of the complex manifold~$\gS_\Om^n$;
this is formula~\eqref{eqppalconv} of Proposition~\ref{propstartingpt}, which
will be our starting point for the proof of Theorem~\ref{thmboundconvga}'.

%%%%%%%%%%%%%%%%%%%%%%%%%%%%%%%%%%%%%%%%%%%%%%%%%%%%%

\begin{nota}
Given $\ze\in\gS_\Om$, we denote by
$
\cL_\ze \col \D_{R_\Om(\ze)} \to \gS_\Om
$
the holomorphic map defined by
\[
\cL_\ze(\xi) \defeq \text{endpoint of the lift which starts at~$\ze$ of the path
$t\in[0,1] \mapsto \dze + t \xi$}
\]
(so that $\cL_\ze(\xi)$ can be thought of as ``the lift of $\dze+\xi$ wich
sits on the same sheet of~$\gS_\Om$ as~$\ze$'').
We shall often use the shorthand 
\[ \ze+\xi \defeq \cL_\ze(\xi) \]
(beware that, in the latter formula, $\xi\in\D_{R_\Om(\ze)}$ is a complex number
but not~$\ze$ nor $\ze+\xi$, which are points of~$\gS_\Om$).
If $n\ge1$ and $\un\ze = (\ze_1,\ldots,\ze_n) \in \gS_\Om^n$, we also set
\begin{align*}
S_n(\un\ze) &\defeq \dze_1 + \cdots + \dze_n \in\C, \\[1ex]
\cL_{\un\ze}(\un\xi) &\defeq \un\ze + \un\xi \defeq
\big( \cL_{\ze_1}(\xi_1), \ldots, \cL_{\ze_n}(\xi_n) \big)
\in \gS_\Om^n
\end{align*}
for $\un\xi = (\xi_1,\ldots,\xi_n) \in \C^n$ close enough to~$0$
(it suffices that $|\xi_j| < R_\Om(\ze_j)$;
observe that $S_n(\un\ze+\un\xi) = S_n(\un\ze)+\xi_1+\cdots+\xi_n$).
\end{nota}

%%%%%%%%%%%%%%%%%%%%%%%%%%%%%%%%%%%%%%%%%%%%%%%%%%%%%

\begin{nota}	\label{notesimplexn}
For any $n\ge1$, we denote by~$\De_n$ the $n$-dimensional
simplex\footnote{
This $\De_n$ has nothing to do with \'Ecalle's alien derivatives~$\De_\om$ mentioned earlier.
}
\[
\De_n \defeq \{\, (s_1,\ldots,s_n)\in\R^n \mid s_1,\ldots,s_n\ge0
\;\text{and}\;
s_1 +\cdots + s_n \le 1 \,\}
\]
with the standard orientation, and by $[\De_n]\in\gE_n(\R^n)$ the corresponding integration
current:
\[
[\De_n] \col \text{$\al$ complex-valued smooth $n$-form on $\R^n$} \mapsto \int_{\De_n} \al \in \C.
\]
For every $\ze\in\D_{\rho(\Om)}$, we consider the map
\[
\un\gD(\ze) \col
\un s = (s_1,\ldots,s_n) \mapsto 
\un\gD(\ze,\un s) \defeq {0_\Om}+(s_1\ze,\ldots,s_n\ze) \in \gS_\Om^n,
\]
defined in a neighbourhood of~$\De_n$ in~$\R^n$,
and denote by $\un\gD(\ze)_\# [\De_n] \in \gE_n(\gS_\Om^n)$ the push-forward
of~$[\De_n]$ by~$\un\gD(\ze)$:
\[
\un\gD(\ze)_\# [\De_n] \col \text{$\be$ smooth $n$-form on $\gS_\Om^n$} \mapsto
[\De_n]\big( \un\gD(\ze)^\#\be \big).
\]
\end{nota}

%%%%%%%%%%%%%%%%%%%%%%%%%%%%%%%%%%%%%%%%%%%%%%%%%%%%%

See Appendix~\ref{appNorCur} for our notations in relation with currents.
Notice that the last formula makes sense because $\un\gD(\ze)$ is a smooth map,
thus the pullback form $\un\gD(\ze)^\#\be$ is well-defined in a neighbourhood of~$\De_n$.
The reason for using the language of currents and Geometric Measure Theory is
that later we shall require the push-forward of integration currents by Lipschitz
maps which are not smooth everywhere.
The reader is referred to Appendix~\ref{appNorCur} for a survey of a few facts
of the theory which will be useful for us.

%%%%%%%%%%%%%%%%%%%%%%%%%%%%%%%%%%%%%%%%%%%%%%%%%%%%%

\begin{prop}	\label{propstartingpt}
For $\hat\ph_1,\ldots,\hat\ph_n \in \hat\gR_\Om$ and $\ze \in \D_{\rho(\Om)}$, one has
\beglabel{eqppalconv}
1 * \hat\ph_1 * \cdots * \hat\ph_n(\ze) = 
\un\gD(\ze)_\# [\De_n](\be)
\quad \text{with $\be =
\hat\ph_1(\ze_1) \cdots \hat\ph_n(\ze_n) 
\, \dd\ze_1 \wedge \cdots \wedge \dd\ze_n$},
\elabel
where we denote by $\dd\ze_1 \wedge \cdots \wedge \dd\ze_n$ the pullback
in $\gS_\Om^n$ by 
$\pi_\Om^{\otimes n} \col \un\ze \in \gS_\Om^n \mapsto
\un\xi = \big( \dze_1,\ldots,\dze_n \big)$
of the $n$-form $\dd\xi_1 \wedge \cdots \wedge \dd\xi_n$ of~$\C^n$. 
\end{prop}

%%%%%%%%%%%%%%%%%%%%%%%%%%%%%%%%%%%%%%%%%%%%%%%%%%%%%

\begin{proof}
Since $(\ze_1,\ldots,\ze_n)\in\D_{\rho(\Om)}^n \mapsto {0_\Om}+(\ze_1,\ldots,\ze_n)
\in \gS_\Om^n$ is an analytic chart which covers a neighbourhood of $\un\gD(\ze)(\De_n)$,
we can write $\un\gD(\ze)^\#\be = \hat\ph_1(s_1\ze) \cdots \hat\ph_n(s_n\ze) \dd
s_1  \wedge \cdots \wedge \dd s_n$. Since
\[
\De_n = \{\, (s_1,\ldots,s_n)\in\R^n \mid s_1\in[0,1],
s_2 \in [0,1-s_1], \ldots,
s_n \in [0,1-(s_1+\cdots+s_{n-1})] \,\}
\]
with the standard orientation,
the \rhs\ of the identity stated in~\eqref{eqppalconv} can be rewritten
\begin{gather}
\notag
\ze^n \int_0^1 \dd s_1\int_0^{1-s_1} \dd s_2 \cdots
\int_0^{1-(s_1+\cdots+s_{n-1})} \dd s_n \,
\hat\ph_1(s_1\ze) \cdots \hat\ph_n(s_n\ze) \\[-2ex]
\intertext{or}
\label{eqdefconvmultint}
\int_0^\ze \dd \ze_1\int_0^{\ze-\ze_1} \dd \ze_2 \cdots
\int_0^{\ze-(\ze_1+\cdots+\ze_{n-1})} \dd \ze_n \,
\hat\ph_1(\ze_1) \cdots \hat\ph_n(\ze_n).
\end{gather}
When $n=1$, formula~\eqref{eqppalconv} is thus the very definition of $1*\hat\ph_1(\ze)$. Writing
\[
1 * \hat\ph_1 * \cdots * \hat\ph_n(\ze) =
\int_0^\ze \dd \ze_1 \, \hat\ph_1(\ze_1) \big(1*\hat\ph_2*\cdots*\hat\ph_n\big)(\ze-\ze_1),
\]
we get the general case by induction.
\end{proof}

\newpage

%%%%%%%%%%%%%%%%%%%%%%%%%%%%%%%%%%%%%%%%%%%%%%%%%%%%%
%%%%%%%%%%%%%%%%%%%%%%%%%%%%%%%%%%%%%%%%%%%%%%%%%%%%%

\section{Deformation of the $n$-dimensional integration current in $\gS_\Om^n$}
\label{secdefor}

%%%%%%%%%%%%%%%%%%%%%%%%%%%%%%%%%%%%%%%%%%%%%%%%%%%%%
%%%%%%%%%%%%%%%%%%%%%%%%%%%%%%%%%%%%%%%%%%%%%%%%%%%%%

% \emph{In this section, we fix an interval $J = [a,b]$ and a path $\ga \col J
% \mapsto \C\setminus\Om$ which is $C^1$ on subintervals $J_k$,
% $k=0,\ldots,N-1$, forming a subdivision of~$J$ (\ie $a=t_0 < t_1 < \cdots < t_N
% = b$) and such that $\ga(a) \in \D_{\rho(\Om)}^*$.}

In this section, we fix an interval $J = [a,b]$ and a path $\ga \col J
\mapsto \C\setminus\Om$ such that $\ga(a) \in \D_{\rho(\Om)}^*$;
we denote by~$\ti\ga$ the lift of~$\ga$ which starts in the principal sheet of~$\gS_\Om$.
In order to obtain the analytic continuation of formula~\eqref{eqppalconv}, we
shall deform the $n$-dimensional integration current $\un\gD(\ze)_\# [\De_n]$ as indicated in
Proposition~\ref{propcontgaconvolPsi} below.

\begin{Def}
Given $n\ge1$, for $\ze \in \C$ and $j = 1,\ldots,n$, we set
\[
\cN(\ze) \defeq \{\, \un\ze \in \gS_\Om^n \mid S_n(\un\ze) = \ze \,\},
\qquad
\cN_j \defeq \{\, \un\ze=(\ze_1,\ldots,\ze_n) \in \gS_\Om^n \mid \ze_j = 0_\Om \,\}.
\]
We call \emph{$\ga$-adapted origin-fixing isotopy in $\gS_\Om^n$}
any family $(\Psi_t)_{t\in J}$ of homeomorphisms of $\gS_\Om^n$ such that 
$\Psi_a = \ID$, 
the map
$(t,\un\ze) \in J \times \gS_\Om^n \mapsto
\Psi_t(\un\ze) \in \gS_\Om^n$
is locally Lipschitz,\footnote{%
By that, we mean that each point of $\gS_\Om^n$ admits an open
neighbourhood~$\gU$ on which 
$\pi_\Om^{\otimes n} \col \gS_\Om^n \to \C^n$
induces a biholomorphism and such that the map
$(t,\un\xi) \in J \times \pi_\Om^{\otimes n}(\gU) \mapsto
\pi_\Om^{\otimes n}\circ \Psi_t\circ 
\big( (\pi_\Om^{\otimes n})_{|\gU} \big)\ii(\un\xi) \in\C^n$
is Lipschitz.
}
and for any $t\in J$ and $j=1,\ldots,n$,
\begin{align*}
\un\ze \in \cN\big(\ga(a)\big) \quad & \Rightarrow \quad
\Psi_t(\un\ze) \in \cN\big(\ga(t)\big), \\
\un\ze \in \cN_j \quad & \Rightarrow \quad
\Psi_t(\un\ze) \in \cN_j.
\end{align*}
%
% and its restriction to
%
% $\{\, (t,\un\ze) \in J_k\times \gS_\Om^n \mid
%
% \un\ze \notin \cN\big(\ga(t)\big) \cup \cN_1 \cup \ldots \cup \cN_n \,\}$ 
%
% is $C^1$ for each~$k$.
%
\end{Def}

%%%%%%%%%%%%%%%%%%%%%%%%%%%%%%%%%%%%%%%%%%%%%%%%%%%%%

\begin{prop}	\label{propcontgaconvolPsi}
Suppose that $(\Psi_t)_{t\in J}$ is a $\ga$-adapted origin-fixing isotopy in
$\gS_\Om^n$.
Then, for any $\hat\ph_1,\ldots,\hat\ph_n \in \hat\gR_\Om$,
the analytic continuation of
$1 * \hat\ph_1 * \cdots * \hat\ph_n$ along~$\ga$ is given by
\beglabel{eqcontgaconvolPsi}
%
% \cont_\ga|t}
%
\big(1 * \hat\ph_1 * \cdots * \hat\ph_n\big)\big(\ti\ga(t)\big) 
= \big(\Psi_t \circ \un\gD(\ga(a)) \big)_\#[\De_n] (\be),
\qquad t \in J,
\elabel
with $\be =
\hat\ph_1(\ze_1) \cdots \hat\ph_n(\ze_n) 
\, \dd\ze_1 \wedge \cdots \wedge \dd\ze_n$.
\end{prop}

%%%%%%%%%%%%%%%%%%%%%%%%%%%%%%%%%%%%%%%%%%%%%%%%%%%%%

\begin{figure}
\begin{center}
\includegraphics[height=2.9in]{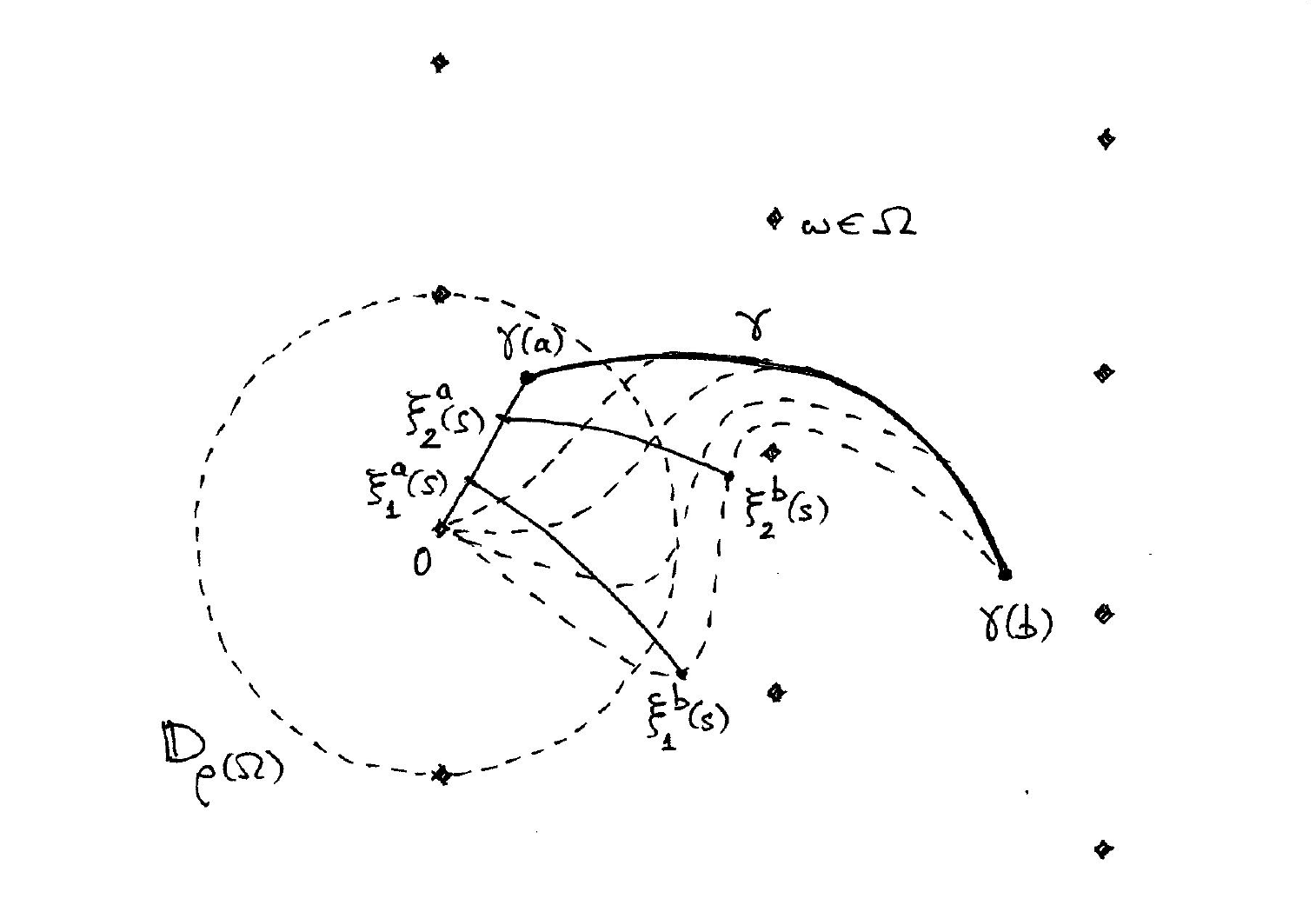}

\caption{Projections of
$\big(\xi^t_1(\un s), \ldots, \xi^t_n(\un s)\big) 
\defeq \Psi_t\big( s_1\ga(a), \ldots, s_n\ga(a) \big) 
= \Psi_t\circ\un\gD\big(\ga(a)\big)(\un s)$.}

\label{figProjDefCell}

\end{center}
\end{figure}

%%%%%%%%%%%%%%%%%%%%%%%%%%%%%%%%%%%%%%%%%%%%%%%%%%%%%

See Figure~\ref{figProjDefCell}.
Observe that, for each $t\in J$, the map $\Psi_t\circ\un\gD\big(\ga(a)\big)
\col \De_n \to \gS_\Om^n$ is Lipschitz, so that the push-forward 
$\big(\Psi_t \circ \un\gD(\ga(a)) \big)_\#[\De_n]$ is a well-defined
$n$-dimensional current of~$\gS_\Om^n$ (see Appendix~\ref{appNorCur}).
The proof of Proposition~\ref{propcontgaconvolPsi} relies on the following more
general statement:
\begin{nota}
Given a map
$\un C=(C_1,\ldots,C_n) \col J\times\De_n \to \gS_\Om^n$,
for each $t\in J$ we denote by $\un C_t \col \De_n \to \gS_\Om^n$ the partial
map defined by
\[
\un s\in\De_n \mapsto \un C_t(\un s) \defeq \un C(t,\un s)
\]
(not to be confused with the components $C_j \col J\times\De_n \to \gS_\Om$, $j=1,\ldots,n$).
\end{nota}

%%%%%%%%%%%%%%%%%%%%%%%%%%%%%%%%%%%%%%%%%%%%%%%%%%%%%

\begin{prop}	\label{propdeformalgD}
Let $\be$ be a holomorphic $n$-form on~$\gS_\Om^n$ and
\[
F \col \ze\in\D_{\rho(\Om)} \mapsto F(\ze) \defeq \un\gD(\ze)_\#[\De_n](\be).
\]
Then $F$ is a holomorphic function in~$\D_{\rho(\Om)}$.

Let $\un C \col J\times\De_n \to\gS_\Om^n$ be a Lipschitz map%
\footnote{%
in the sense that $\pi_\Om^{\otimes n}\circ \un C \col J\times\De_n \to\C^n$ is Lipschitz
}
%
%$C^1$ in restriction to $J_k\times\INT\De_n$ for each~$k$
%
such that the partial map corresponding to $t=a$ satisfies
\[
\un C_a = \un\gD\big(\ga(a)\big)
\]
and that, for every $t\in J$, $\un s = (s_1,\ldots, s_n)\in\De_n$ and $j=1,\ldots,n$,
\begin{align*}
s_1 + \cdots + s_n = 1 \quad & \Rightarrow \quad
\un C(t,\un s) \in \cN\big(\ga(t)\big) \\
% S_n\circ\un C(t,\un s) = \ga(t) \\
%
s_j = 0 \quad \quad \quad & \Rightarrow \quad \quad
\un C(t,\un s) \in \cN_j.
% C_j(t,\un s) = 0_\Om.
%
\end{align*}
Then $F$ admits analytic continuation along~$\ga$ and, for each $t\in J$,
\beglabel{eqcontgaF}
%
% \cont_{\ga|t} 
%
F\big(\ti\ga(t)\big) = (\un C_t)_\#[\De_n](\be).
\elabel
%
% where $\un C_t$ is the partial map viewed as an $n$-dimensional cell.
%
\end{prop}

%%%%%%%%%%%%%%%%%%%%%%%%%%%%%%%%%%%%%%%%%%%%%%%%%%%%%

The proof of Proposition~\ref{propdeformalgD} requires the following consequence
of the Cauchy-Poincar\'e Theorem \cite{Shabat}:

%%%%%%%%%%%%%%%%%%%%%%%%%%%%%%%%%%%%%%%%%%%%%%%%%%%%%

\begin{lem}	\label{lemconsCP}
Let $M$ be a complex analytic manifold of dimension~$n$ and let
$N_0,N_1,\ldots,N_n$ be complex analytic hypersurfaces of~$M$.
Let ${\un H} \col [0,1]\times \De_n \to M$ be a Lipschitz map such that,
% $C^1$ on $[0,1]\times\INT\De_n$
%
for every $\tau\in[0,1]$, $\un s = (s_1,\ldots, s_n)\in\De_n$ and $j=1,\ldots,n$,
\begin{align*}
s_1 + \cdots + s_n = 1 \quad & \Rightarrow \quad
{\un H}(\tau,s) \in N_0 \\
s_j = 0 \quad \quad \quad & \Rightarrow \quad
{\un H}(\tau,s) \in N_j.
\end{align*}
Then the partial maps ${\un H}_0$ and ${\un H}_1$ corresponding to $\tau = 0$
and $\tau=1$ satisfy
\beglabel{eqhomotopformul} 
({\un H}_0)_\#[\De_n](\be) = ({\un H}_1)_\#[\De_n](\be)
\elabel
for any holomorphic $n$-form~$\be$ on~$M$.
\end{lem}

%%%%%%%%%%%%%%%%%%%%%%%%%%%%%%%%%%%%%%%%%%%%%%%%%%%%%

\begin{proof}[Proof of Lemma~\ref{lemconsCP}]
Let $\be$ be a holomorphic $n$-form on~$M$.
Let us consider $P \defeq [0,1]\times \De_n$ and the corresponding
$(n+1)$-dimensional integration current
$[P] \in \gE_{n+1}(\R^{n+1})$.
Its boundary can be written 
\[
\pa[P] = Q_1 - Q_0 + B_0 + \cdots + B_n,
\]
where each summand is an $n$-dimensional current with compact support:
\[
\spt Q_i = \{i\}\times\De_n, \qquad
\spt B_j = [0,1]\times F_j
\]
with $F_j \defeq$ the face of~$\pa\De_n$ defined by $s_j=0$ if $j\ge1$ or
$s_1+\cdots+s_n=1$ if $j=0$.
This is a simple adaptation of formula~\eqref{eqpaDenaffine} of
Appendix~\ref{appNorCur}; in fact, 
$Q_i = [A_i(\De_n)]$ with an affine map $A_i \col \un x \in \R^n \mapsto (i,\un x) \in
\R^{n+1}$ 
and $B_j = \pm [A_j^*(\De_n)]$ with some other injective affine maps
$A_j^* \col \R^n \to \R^{n+1}$ mapping~$\De_n$ to $[0,1]\times F_j$.
In this situation, according to Lemma~\ref{lempushcomm} and
formula~\eqref{eqdefphifiesiszeroDe}, we have
\[
\pa \un H_\# [P] = \un H_\# \pa [P],
\qquad
\un H_\# Q_i = (\un H\circ A_i)_\# [\De_n],
\qquad
\un H_\# B_j = (\un H\circ A_j^*)_\# [\De_n].
\]
On the one hand, the Cauchy-Poincar\'e Theorem tells us that $\pa\un
H_\#[P](\be)=0$ (because $\dd\be=0$), and $H\circ A_i = H_i$, 
thus 
\[
(\un H_0)_\# [\De_n](\be) - (\un H_1)_\# [\De_n](\be) =
\un H_\# B_0(\be) + \cdots + \un H_\# B_n(\be).
\]
On the other hand $\spt \un H_\# B_j \subset N_j$ and
the restriction of~$\be$ to any complex hypersurface vanishes identically
(because it is a holomorphic form of maximal degree), 
thus $\un H_\# B_j(\be) = 0$,
and \eqref{eqhomotopformul} is proved.
%
% We may consider~${\un H}$ as an $(n+1)$-dimensional cell, the boundary of which is of
% the form
% %
% \[
% %
% \pa {\un H} = {\un H}_1 - {\un H}_0 + {\un B}_0 + {\un B}_1 + \cdots + {\un B}_n,
% %
% \]
% %
% with $n$-dimensional cells ${\un B}_0,{\un B}_1,\ldots,{\un B}_n$ obtained by restriction of~${\un H}$
% to the faces $\{ s_1 + \cdots + s_n = 1 \}$ or $\{ s_j = 0 \}$ of~$\De_n$;
% %
% in particular, each cell~${\un B}_j$ has its image contained in~$N_j$.
%
%
% \[
%
% 0 = \int_{{\un H}_1} \be - \int_{{\un H}_0} \be + \sum_{j=0}^n \int_{{\un B}_j} \be.
%
% \]
%
% But the integral of a holomorphic $n$-form on a cell which is contained in a
% complex analytic hypersurface vanishes too, hence the conclusion follows.
%
\end{proof}

%%%%%%%%%%%%%%%%%%%%%%%%%%%%%%%%%%%%%%%%%%%%%%%%%%%%%

\begin{proof}[Proof of Proposition~\ref{propdeformalgD}]
Observe that the function~$R_\Om$ defined by~\eqref{eqdefRze} is continuous, thus we
can define a positive number
\[
R^* \defeq \min\big\{ R_\Om\big(C_j(t,\un s)\big) \mid 
t\in J, \; \un s\in \De_n, \; j=1,\ldots,n \big\}
\]
and, for each $t\in J$ and $\ze \in D\big(\ga(t),R^*\big)$, 
% an $n$-dimensional cell
%
a Lipschitz map and a complex number
\[
\un\gD_t(\ze) \col
\un s \in \De_n \mapsto
{\un C(t,\un s)} + \big(\ze-\ga(t)\big) \un s
\in \gS_\Om^n,
\qquad
G_t(\ze) \defeq \un\gD_t(\ze)_\#[\De_n](\be).
\]
For $\ze\in D\big(\ga(a),R^*\big)$, we have $\un\gD_a(\ze) = \un\gD(\ze)$, hence
$G_a(\ze) = F(\ze)$.
For $t\in J$, we have $\un\gD_t\big(\ga(t)\big) = \un C_t$, hence
\[
G_t\big(\ga(t)\big) = (\un C_t)_\#[\De_n](\be).
\]
Therefore it suffices to show that, for each $t\in J$, % for each $k$ and for each $t\in J_k$,
\begin{enumerate}[{i)}]
\item
the function $G_t$ is holomorphic in $D\big(\ga(t),R^*\big)$
(and $G_a=F$ is even holomorphic in $\D_{\rho(\Om)}$);
\item
for any $t'\in J$ close enough to~$t$, the functions~$G_t$ and~$G_{t'}$ coincide
in a neighbourhood of~$\ga(t)$.
\end{enumerate}

%%%%%%%%%%%%%%%%%%%%%%%%%%%%%%%%%%%%%%%%%%%%%%%%%%%%%

\medskip

\noindent \textbf{i)}
The case of $G_a=F$ is easier because, for $\ze\in\D_{\rho(\Om)} \cup
D\big(\ga(a),R^*\big)$, the range of $\un\gD_a(\ze) =
\un\gD(\ze)$ entirely lies in a domain $\gU = \gU_1\times\cdots\gU_n$, where
each $\gU_j$ is an open subset of~$\gS_\Om$ in restriction to which~$\pi_\Om$ is
injective, so that 
\beglabel{eqdefchigU}
\chi = \pi_\Om^{\otimes n} \col 
(\ze_1,\ldots,\ze_n) \in\gU \mapsto (\xi_1,\ldots,\xi_n) = 
\big(\dze_1,\ldots,\dze_n\big)
\elabel
is an analytic chart of~$\gS_\Om^n$;
we can write $\chi^\#\be = f(\xi_1,\ldots,\xi_n) \, \dd\xi_1 \wedge \cdots
\wedge \dd\xi_n$ with a holomorphic function~$f$ and
$\chi\circ \un\gD_a(\ze)(\un s) = (s_1\ze,\ldots,s_n\ze)$, therefore
\[
G_a(\ze) = F(\ze) = \ze^n \int_{\De_n} f(s_1\ze,\ldots,s_n\ze)
\, \dd s_1 \cdots \dd s_n
\]
is holomorphic.

Given $t\in J$, by compactness,
we can cover~$\De_n$ by simplices $Q\sst m$, $1\le m \le M$, so that any
intersection $Q\sst m\cap Q\sst{m'}$ is contained in an affine hyperplane
of~$\R^n$
and each~$Q\sst m$ is small enough for 
$\bigcup_{\ze \in D(\ga(t),R^*)}\un\gD_t(\ze)\big(Q\sst m\big)$ 
to be contained in the domain~$\gU\sst m$ of an analytic chart~$\chi\sst
m$ similar to~\eqref{eqdefchigU} (\ie $\gU\sst m$ is a product of factors on
which~$\pi_\Om$ is injective and~$\chi\sst m$ is defined by the same formula
as~$\chi$ but on~$\gU\sst m$).
For each~$m$, 
we can write $\big(\chi\sst m\big)^\#\be = f\sst m(\xi_1,\ldots,\xi_n) \, \dd\xi_1 \wedge \cdots
\wedge \dd\xi_n$ with a holomorphic function~$f\sst m$ and
$\chi\sst m\circ \un\gD_t(\ze) = 
\big(\xi_1\sst m(\ze,\,\cdot\,),\ldots,\xi_n\sst m(\ze,\,\cdot\,)\big)$
with, for each $j=1,\ldots,n$,
\[
(\ze,\un s) \in D\big(\ga(t),R^*\big) \times Q\sst m
\mapsto
\xi_j\sst m(\ze,\un s) = \pi_\Om\circ C_j(t,\un s) + s_j \big(\ze-\ga(t)\big).
\]
These functions~$\xi_j\sst m$ are holomorphic in~$\ze$;
applying Rademacher's theorem to $\un s \mapsto \pi_\Om\circ C_j(t,\un s)$
(recall that $t$ is fixed), we see that, for almost every $\un s$, the partial
derivatives of~$\xi_j\sst m$ exist and are holomorphic in~$\ze$, therefore
\[
G_t(\ze) = \sum_{m=1}^M \int_{Q\sst m} 
f\sst m\big(\xi_1\sst m(\ze,\un s),\ldots,\xi_n\sst m(\ze,\un s)\big)
\Det\bigg[ \frac{\pa\xi_i\sst m}{\pa s_j\,}(\ze,\un s)\bigg]_{1\le i,j\le n}
\, \dd s_1\cdots \dd s_n
\]
is holomorphic for $\ze \in D\big(\ga(t),R^*\big)$.

\medskip

\noindent \textbf{ii)}
We now fix $t\in J$. By compactness, for $t'\in J$ close enough to~$t$, we
can write
\[
\un C(t',\un s) = {\un C(t,\un s)} + \un\de(\un s)
\]
for all $\un s \in \De_n$, with
\[
\de_j(\un s) \defeq \pi_\Om \big( C_j(t',\un s) - C_j(t,\un s) \big) 
\in \D_{\frac{R^*}{2n}},
\qquad j = 1,\ldots, n.
\]
Then $\ga(t') \in D\big(\ga(t),R^*/2\big)$ (because $s_1+\cdots+s_n=1$ implies
$S_n\circ\de(\un s) = \ga(t')-\ga(t)$) and, 
for $\ze \in D\big(\ga(t'),R^*/2\big)$, we have
\[
G_t(\ze) \defeq \un\gD_t(\ze)_\#[\De_n](\be),
\qquad
G_{t'}(\ze) \defeq \un\gD_{t'}(\ze)_\#[\De_n](\be)
\]
with
\[
\un\gD_t(\ze)(\un s) =
{\un C(t,\un s)} + \big(\ze-\ga(t)\big) \un s,
\qquad
\un\gD_{t'}(\ze)(\un s) =
{\un C(t,\un s)} + \un\de(s) + \big(\ze-\ga({t'})\big) \un s.
\]
Let us define a Lipschitz map ${\un H} \col [0,1]\times \De_n \to \gS_\Om^n$ by
\[
\un H(\tau,\un s) \defeq
{\un C(t,\un s)} + (1-\tau) \big(\ze-\ga(t)\big) \un s +
\tau\big(\un\de(s) + \big(\ze-\ga({t'})\big) \un s\big).
\]
An easy computation yields
\begin{align*}
s_1 + \cdots + s_n = 1 \quad & \Rightarrow \quad
S_n\circ\un H(\tau,\un s) = \ze \\
s_j = 0 \quad \quad \quad & \Rightarrow \quad \quad
H_j(\tau,\un s) = 0_\Om.
\end{align*}
We can thus apply Lemma~\ref{lemconsCP} with 
$N_0 = \cN(\ze)$
and $N_j = \cN_j$,
and conclude that $G_t \equiv G_{t'}$ on $D\big(\ga(t'),R^*/2\big)$.
\end{proof}

%%%%%%%%%%%%%%%%%%%%%%%%%%%%%%%%%%%%%%%%%%%%%%%%%%%%%

\begin{proof}[Proof of Proposition~\ref{propdeformalgD}]
In view of Proposition~\ref{propstartingpt}, we can apply
Proposition~\ref{propdeformalgD} with 
$\be = \hat\ph_1(\ze_1) \cdots \hat\ph_n(\ze_n) 
\, \dd\ze_1 \wedge \cdots \wedge \dd\ze_n$
and $\un C_t = \Psi_t \circ \un\gD\big(\ga(a)\big)$.
\end{proof}

%%%%%%%%%%%%%%%%%%%%%%%%%%%%%%%%%%%%%%%%%%%%%%%%%%%%%
%%%%%%%%%%%%%%%%%%%%%%%%%%%%%%%%%%%%%%%%%%%%%%%%%%%%%

\section{Construction of an adapted origin-fixing isotopy in $\gS_\Om^n$}

\label{secConstrIsot}

%%%%%%%%%%%%%%%%%%%%%%%%%%%%%%%%%%%%%%%%%%%%%%%%%%%%%
%%%%%%%%%%%%%%%%%%%%%%%%%%%%%%%%%%%%%%%%%%%%%%%%%%%%%

To prove Theorem~\ref{thmboundconvga}, formula~\eqref{eqrelderiv} tells us that it
is sufficient to deal with the analytic continuation of products of the form
$1*\hat\ph_1*\cdots*\hat\ph_n$ instead of $\hat\ph_1*\cdots*\hat\ph_n$ itself,
and Proposition~\ref{propcontgaconvolPsi} tells us that, to do so, we only need to
construct explicit $\ga$-adapted origin-fixing isotopies $(\Psi_t)$ and to provide
estimates.

This section aims at constructing $(\Psi_t)$ for any given $C^1$
path~$\ga$ (estimates are postponed to Section~\ref{secestimisot}).
Our method is inspired by an appendix of \cite{CNP} and
is a generalization of Section~6.2 of~\cite{stabiconv}.

%%%%%%%%%%%%%%%%%%%%%%%%%%%%%%%%%%%%%%%%%%%%%%%%%%%%%

\begin{prop}	\label{propisotopyflow}
Let $\ga \col J = [a,b] \to \C\setminus\Om$ be a $C^1$ path such that
$\ga(a) \in \D_{\rho(\Om)}^*$,
and let $\eta \col \C \to [0,+\infty)$ be a locally Lipschitz function such that
\[
\{\, \xi\in \C \mid \eta(\xi) = 0 \,\} = \Om.
\]
Then the function
\beglabel{eqdefD}
(t,\un\ze) \in J \times \gS_\Om^n \mapsto D(t,\un\ze) \defeq
\eta(\dze_1) + \cdots + \eta(\dze_n) + 
\eta\big( \ga(t) - S_n(\un\ze) \big)
\elabel
is everywhere positive and the formula
\beglabel{eqdefX}
X(t,\un\ze) =  \left| \begin{aligned}
X_1 &\defeq \frac{\eta(\dze_1)}{D(t,\un\ze)} \ga'(t) \\
& \qquad \vdots \\[1ex]
X_n &\defeq \frac{\eta(\dze_n)}{D(t,\un\ze)} \ga'(t)
\end{aligned} \right.
\elabel
defines a non-autonomous vector field 
$X(t,\un\ze) \in T_{\un\ze} \big( \gS_\Om^n \big) \simeq \C^n$
(using the canonical identification between the tangent space of~$\gS_\Om$ at
any point and~$\C$ provided by the tangent map of the local biholomorphism~$\pi_\Om$)
which admits a flow map~$\Psi_t$ between time~$a$ and time~$t$
for every $t\in J$
and induces a $\ga$-adapted origin-fixing isotopy $(\Psi_t)_{t\in J}$ in $\gS_\Om^n$.
\end{prop}

%%%%%%%%%%%%%%%%%%%%%%%%%%%%%%%%%%%%%%%%%%%%%%%%%%%%%

An example of function which satisfies the assumptions of
Proposition~\ref{propisotopyflow} is
\[
\eta(\xi) \defeq \dist(\xi,\Om), \qquad \xi\in\C.
\]

%%%%%%%%%%%%%%%%%%%%%%%%%%%%%%%%%%%%%%%%%%%%%%%%%%%%%

\begin{proof}[Proof of Proposition~\ref{propisotopyflow}]
%
% Let us denote by $J_k =[t_k,t_{k+1}]$, $k=0,\ldots,N-1$, subintervals forming a subdivision
% of~$J$ so that $\ga$ is $C^1$ on each~$J_k$.
%
% \medskip
%
%\noindent 
\textbf{(a)}
Observe that $D\big(t,(\ze_1,\ldots,\ze_n)\big) = 
\ti D\big(t,(\dze_1,\ldots,\dze_n)\big)$
with
\beglabel{eqdeftiDtunze}
(t,\un\xi) \in J \times \C^n \mapsto \ti D(t,\un\xi) \defeq
\eta(\xi_1) + \cdots + \eta(\xi_n) + 
\eta\big( \ga(t) - \ti S_n(\un\xi) \big)
\elabel
and $\ti S_n(\un\xi) \defeq \xi_1+\cdots+\xi_n$ for any $\un\xi\in\C^n$.
The function $\ti D$ is everywhere positive: suppose indeed $\ti D(t,\un\xi) = 0$
with $t\in J$ and $\un\xi \in \C^n$, we would have 
\[
\xi_1, \ldots, \xi_n, \ga(t) - \ti S(\un\xi) \in \Om,
\]
whence $\ga(t) \in \Om$ by the stability under addition of~$\Om$, but this is
contrary to the hypothesis on~$\ga$.

Therefore $D>0$, the vector field~$X$ is well-defined and in fact
\[
X\big(t,(\ze_1,\ldots,\ze_n)\big) = 
\ti X\big(t,(\dze_1,\ldots,\dze_n)\big)
\]
with a non-autonomous vector field~$\ti X$ defined in $J \times \C^n$, the
components of which are
\beglabel{eqdefXjtunze}
\ti X_j(t,\un\xi) \defeq \frac{\eta(\xi_j)}{\ti D(t,\un\xi)} \ga'(t),
\qquad j= 1,\ldots,n.
\elabel
These functions are locally Lipschitz on $J\times\C^n$,
thus we can apply the Cauchy-Lipschitz theorem on the existence and uniqueness of
solutions to differential equations to $\dd\un\xi / \dd t = \ti X(t,\un\xi)$:
for every $t^*\in J_k$ and $\un\xi\in\C^n$, there is a unique maximal solution
$t\mapsto \ti\Phi^{t^*,t}(\un\xi)$ such that 
$\ti\Phi^{t^*,t^*}(\un\xi) = \un\xi$.
The fact that the vector field~$\ti X$ is bounded implies that
$\ti\Phi^{t^*,t}(\un\xi)$ is defined for all $t\in J$ and the classical theory
guarantees that $(t^*,t,\un\xi) \mapsto \ti\Phi^{t^*,t}(\un\xi)$ is locally Lipschitz on
$J\times J\times\C^n$. % and $C^1$ on $J_k\times J_k\times(\C\setminus\Om)^n$.

\medskip

\noindent \textbf{(b)}
For each $\om\in\Om$ and $j=1,\ldots,n$, we set
\[
\ti\cN_j(\om) \defeq \{\, \un\xi=(\xi_1,\ldots,\xi_n) \in \C^n \mid \xi_j = \om \,\}.
\]
We have $\ti X_j \equiv 0$ on $J\times\ti\cN_j(\om)$, thus $\ti\Phi^{t^*,t}$ leaves
$\ti\cN_j(\om)$ invariant for every $(t^*,t)\in J\times J$.
In particular, since $0\in\Om$,
\beglabel{eqPhicNj}
\un\xi \in \ti\cN_j(0) \quad\Rightarrow\quad 
\ti\Phi^{t^*,t}(\un\xi) \in \ti\cN_j(0).
\elabel
The non-autonomous flow property 
$\ti\Phi^{t,t^*}\circ\ti\Phi^{t^*,t} = \ti\Phi^{t^*,t}\circ\ti\Phi^{t,t^*} = \ID$ 
implies that, for each $(t^*,t)\in J\times J$, $\ti\Phi^{t^*,t}$ is a
homeomorphism the inverse of which is $\ti\Phi^{t,t^*}$, which leaves
$\ti\cN_j(\om)$ invariant,
hence also 
\beglabel{eqPhicNjom}
\un\xi \in \C^n \setminus \ti\cN_j(\om) \quad\Rightarrow\quad 
\ti\Phi^{t^*,t}(\un\xi) \in \C^n \setminus \ti\cN_j(\om).
\elabel
Properties~\eqref{eqPhicNj} and~\eqref{eqPhicNjom} show that the flow map
between times~$t^*$ and~$t$ for~$X$ is well-defined in $\gS_\Om^n$:
for $\un\ze\in \gS_\Om^n$, the solution
$t \mapsto \Phi^{t^*,t}(\un\ze)$ can be obtained as
the lift starting at~$\un\ze$ of the path
$t \mapsto \ti\Phi^{t^*,t}\big( \dze_1, \ldots, \dze_n \big)$
(indeed, each component of this path has its range either reduced to~$\{0\}$ or
contained in $\C\setminus\Om$).

We thus define, for each $t\in J$, a homeomorphism of $\gS_\Om^n$ by
%
% \[
%
$\Psi_t \defeq \Phi^{a,t}$ %\Phi^{t_k,t} \circ \Phi^{t_{k-1},t_k} \circ \cdots \circ \Phi^{t_0,t_1},
%
% \]
%
and observe that % where $k$ is determined by the condition $t\in J_k$, we get
%
% \[
%
$\Psi_t(\cN_j) \subset \cN_j$,
%
% \]
%
$\Psi_a = \ID$ and
$(t,\un\ze)\mapsto\Psi_t(\un\ze)$ is locally Lipschitz on $J\times\gS_\Om^n$.
%
% its restriction to $J_k\times\big(\gS_\Om\setminus\{0_\Om\}\big)^n$ is $C^1$
% for each~$k$.

%%%%%%%%%%%%%%%%%%%%%%%%%%%%%%%%%%%%%%%%%%%%%%%%%%%%%

\medskip

\noindent \textbf{(c)}
It only remains to be proved that 
\beglabel{eqinclusPsit}
\Psi_t\big( \cN(\ga(a)) \big) \subset \Psi_t\big( \cN(\ga(t)) \big)
\elabel
for every $t\in J$. 

Given $\un\ze \in \gS_\Om^n$, the function defined by
\[
\xi_0 \col t \in J \mapsto \ga(t) - S_n\circ\Psi_t(\un\ze).
\]
is $C^1$ on~$J$ and an easy computation yields its derivative in
the form
$\xi_0'(t) = h(t)\ga'(t)/d(t)$, with Lipschitz functions
\[
h(t) \defeq \eta\big(\xi_0(t)\big), \qquad
d(t) \defeq D\big(t,\Psi_t(\un\ze)\big).
\]
Since $\eta$ is Lipschitz on the range of~$\xi_0$, say with Lipschitz
constant~$K$, the function $h=\eta\circ\xi_0$ is Lipschitz on~$J$,
hence its derivative~$h'$ exists almost everywhere on~$J$; 
%
% and $h(t) = \int_{t_k}^t h'$;
%
writing $\abs{h(t')-h(t)} \le K \abs{\xi_0(t')-\xi_0(t)}$, we see that
$\dst \abs{h'(t)} \le K \abs{\xi_0'(t)} \le K {h(t)} \max_{J} \abs*{\tfrac{\ga'}{d}}$
a.e., hence
\[
\text{$g(t) \defeq \frac{h'(t)}{h(t)}$ exists a.e.\ and defines $g \in L^\infty(J)$.}
\]
By the fundamental theorem of Lebesgue integral calculus, $t \mapsto
\int_{a}^t g(\tau)\,\dd\tau$ is differentiable a.e.\ and
\[
h(t) = h(a) \, \exp\Big( \int_{a}^t g(\tau)\,\dd\tau \Big),
\qquad t\in J.
\]

% It follows that, for each $t\in J$, $h(t)$ is the product of~$h(a)$ and a positive factor.
%
Now, if $\un\ze \in \cN\big(\ga(a)\big)$, 
then $\xi_0(a) =0$, thus $h(a)=0$, 
thus $h\equiv 0$ on~$J$, 
thus $\xi_0(t)$ stays in~$\Om$ for all $t\in J$,
thus $\xi_0\equiv 0$ on~$J$,
\ie $\Psi_t(\un\ze) \in \cN\big(\ga(t)\big)$.
\end{proof}

%%%%%%%%%%%%%%%%%%%%%%%%%%%%%%%%%%%%%%%%%%%%%%%%%%%%%
%%%%%%%%%%%%%%%%%%%%%%%%%%%%%%%%%%%%%%%%%%%%%%%%%%%%%

\section{Estimates}	\label{secestimisot}

%%%%%%%%%%%%%%%%%%%%%%%%%%%%%%%%%%%%%%%%%%%%%%%%%%%%%
%%%%%%%%%%%%%%%%%%%%%%%%%%%%%%%%%%%%%%%%%%%%%%%%%%%%%

We are now ready to prove
\begin{thmS}%	\label{thmboundconvga}' 
Let $\de,L>0$ with $\de < \rho(\Om)/2$ and 
\beglabel{eqdefdepC}
\de' \defeq \demi \rho(\Om) \, \ee^{-2L/\de}.
\elabel
Then, for any $n\ge1$ and $\hat\ph_1,\ldots,\hat\ph_n \in \hat\gR_\Om$,
\beglabel{ineqwithone}
\max_{\cK_{\de,L}(\Om)}\abs{ 1*\hat\ph_1*\cdots*\hat\ph_n }
\le \frac{1}{n!} {\big(\rho(\Om)\,\ee^{3L/\de}\big)^n}
\max_{\cK_{\de',L}(\Om)}\abs{\hat\ph_1} \cdots \max_{\cK_{\de',L}(\Om)}\abs{\hat\ph_n}.
\elabel
\end{thmS}

%%%%%%%%%%%%%%%%%%%%%%%%%%%%%%%%%%%%%%%%%%%%%%%%%%%%%

%Notice that $\de'<\de$ as soon as $L\ge\rho(\Om)/4$.
%
The proof of Theorem~\ref{thmboundconvga}' will follow from
%
%%%%%%%%%%%%%%%%%%%%%%%%%%%%%%%%%%%%%%%%%%%%%%%%%%%%%
%%%%%%%%%%%%%%%%%%%%%%%%%%%%%%%%%%%%%%%%%%%%%%%%%%%%%
%
\begin{prop}	\label{propesimisot} 
Let $\de,L>0$. 
Let $\ga \col J = [a,b] \to \C\setminus\Om$ be a $C^1$ path such that $\ga(a)\in
\D^*_{\rho(\Om)/2}$, $\abs{\ga(a)}+b-a \le L$ and
\[
\abs{\ga'(t)} = 1 \ens\text{and}\ens
\dist\big( \ga(t),\Om \big) \ge \de, \qquad
t \in J.
\]
Consider the $\ga$-adapted origin-fixing isotopy $(\Psi_t)_{t\in J}$ defined as in
Proposition~\ref{propisotopyflow} by the flow of the vector field~\eqref{eqdefX}
with the choice $\eta(\xi) = \dist(\xi,\Om)$.
Then, for each $t\in J$, 
\begin{itemize}
\item
the Lipschitz map $\Psi_t\circ \un\gD\big(\ga(a)\big) =
(\xi_1^t,\ldots,\xi_1^t)$ maps $\De_n$ in $\big( \cK_{\de',L}(\Om) \big)^n$,
with $\de'$ as in~\eqref{eqdefdepC},
\item
the almost everywhere defined partial derivatives
$\frac{\pa \dxi^t_i}{\pa s_j} \col \De_n \to \C$ % $T_{\xi^t_i(\un s)} \gS_\Om \simeq\C$
%
% (canonical identification induced by the tangent map of the local biholomorphism~$\pi_\Om$)
%
satisfy
\beglabel{ineqDet}
\abs{\Det\bigg[ \frac{\pa\dxi^t_i}{\pa s_j}(\un s)\bigg]_{1\le i,j\le n}} \le
\big(\rho(\Om)\,\ee^{3L/\de}\big)^n
\quad \text{for a.e.\ $\un s\in\De_n$.}
\elabel
%
%with~$C$ as in~\eqref{eqdefdepC}.
%
\end{itemize}
\end{prop}

%%%%%%%%%%%%%%%%%%%%%%%%%%%%%%%%%%%%%%%%%%%%%%%%%%%%%

\begin{proof}[Proof of Proposition~\ref{propesimisot}]
We first fix $\un s\in \De_n$, omitting it in the notations, and study the solution 
\[
t\in J \mapsto \un\xi^t \defeq
(\xi^t_1,\ldots,\xi^t_n) \defeq \Psi_t\big(\un\gD\big(\ga(a)\big)(\un s)\big)
\]
of the vector field~$X$ defined by~\eqref{eqdefX}, the components of the initial condition being
$\xi^0_i = 0_\Om + s_i \ga(a)$.
\medskip

%%%%%%%%%%%%%%%%%%%%%%%%%%%%%%%%%%%%%%%%%%%%%%%%%%%%%

\noindent \textbf{(a)}
We observe that ${\dd\dxi^t_i}/{\dd t} = X_i(t,\un\xi^t)$ has modulus $\le
1$ for each $i=1,\ldots,n$, thus the path $t\in J \mapsto \dxi^t_i$ has length
$\le b-a$ and stays in~$\D_L$.

\medskip

%%%%%%%%%%%%%%%%%%%%%%%%%%%%%%%%%%%%%%%%%%%%%%%%%%%%%

\noindent \textbf{(b)}
The denominator~\eqref{eqdefD} is
\[
d(t) \defeq D(t,\un\xi^t) \ge \de, \qquad t\in J.
\]
Indeed, we can write 
$d(t) = \eta\big(\dxi^t_0\big) + \eta\big(\dxi^t_1\big) + \cdots +
\eta\big(\dxi^t_n\big)$ with
$\dxi^t_0 \defeq \ga(t) - S_n(\un\xi^t)$,
and, since $\Om$ is stable under addition and $\dxi^t_0+\dxi^t_1+\cdots+\dxi^t_n=\ga(t)$,
the triangle inequality yields
\[
d(t) = \sum_{i=0}^n \dist\big(\dxi^t_i,\Om\big) \ge \dist\big( \ga(t),\Om \big),
\]
which is $\ge \de$ by assumption.
\medskip

%%%%%%%%%%%%%%%%%%%%%%%%%%%%%%%%%%%%%%%%%%%%%%%%%%%%%

\noindent \textbf{(c)}
We now check that for $t\in J$ and $i=1,\ldots,n$,
\beglabel{ineqencadretadxiti}
\ee^{-L/\de} \, \eta\big(\dxi^a_i\big) 
\le \eta\big(\dxi^t_i\big) \le
\ee^{L/\de} \, \eta\big(\dxi^a_i\big).
\elabel
Since $\eta$ is $1$-Lipschitz, the function $h_i \defeq \eta \circ \dxi_i^t$ is
Lipschitz on~$J$ and its derivative exists a.e.; 
writing $\abs{h_i(t')-h_i(t)} \le \abs*{\dxi_i(t')-\dxi_i(t)}$, we see that a.e.
$\abs{h_i'(t)} \le \abs*{\dxi_i'(t)} = {h_i(t)}/d(t)$
hence
$g_i(t) \defeq \frac{h_i'(t)}{h_i(t)}$ exists a.e.\ and defines $g_i \in
L^\infty(J)$ with 
\beglabel{ineqboundgi}
\abs{g_i(t)} \le 1/\de \quad \text{for a.e.\ $t\in J$}.
\elabel
By the fundamental theorem of Lebesgue integral calculus, $t \mapsto
\int_{a}^t g_i(\tau)\,\dd\tau$ is differentiable a.e.\ and
\[
h_i(t) = h_i(a) \, \exp\Big( \int_{a}^t g_i(\tau)\,\dd\tau \Big),
\qquad t\in J,
\]
whence~\eqref{ineqencadretadxiti} follows in view of~\eqref{ineqboundgi}.
\medskip

%%%%%%%%%%%%%%%%%%%%%%%%%%%%%%%%%%%%%%%%%%%%%%%%%%%%%

\noindent \textbf{(d)}
Now, the fact that $\dxi^a_i = s_i \ga(a) \in \D_{\rho(\Om)/2}$ implies that 
$\dist\big( \dxi^a_i,\Om\setminus\{0\} \big) \ge \rho(\Om)/2$, whence
\[
\eta(\dxi^a_i) = \dist\big( \dxi^a_i,\Om \big) = \abs*{\dxi^a_i} \le \rho(\Om)/2.
\]

If $\abs*{\dxi^a_i} < \demi\rho(\Om)\,\ee^{-L/\de}$, then the second inequality
in~\eqref{ineqencadretadxiti} shows that $\eta\big(\dxi^t_i)$ stays
$<\demi\rho(\Om)$, hence $\xi^t_i$ stays in the lift of $\D_{\rho(\Om)/2}$ in
the principal sheet and $R_\Om(\xi^t_i)$ stays $\ge \demi\rho(\Om) > \de'$.

If $\abs*{\dxi^a_i} \ge \demi\rho(\Om)\,\ee^{-L/\de}$, then the first inequality
in~\eqref{ineqencadretadxiti} shows that 
$\eta\big(\dxi^t_i) \ge \demi\rho(\Om)\,\ee^{-2L/\de}$ which equals~$\de'$,
hence $R_\Om(\xi^t_i)$ stays $\ge \de'$.

We infer that $\xi^t_i \in \cK_{\de',L}(\Om)$ for all $t\in J$ in both cases (in view of
point~\textbf{(a)}, since $\xi^t_i\in\gS_\Om$ can be represented by the the
path~$\Ga_s|t \in \gP_\Om$ which is obtained by concatenation of $[0,s_i\ga(a)]$
and
$\tau\in[a,t]\mapsto\dxi^\tau_i$
and has length $\le \abs{\ga(a)} + b-a \le L$).
\medskip

%%%%%%%%%%%%%%%%%%%%%%%%%%%%%%%%%%%%%%%%%%%%%%%%%%%%%

\noindent \textbf{(e)}
It only remains to study the partial derivatives 
$\frac{\pa \xi^t_i}{\pa s_j}(\un s)$
which, given $t\in J$, exist for almost every $\un s\in\De_n$ by virtue of Rademacher's theorem.
We first prove that
for every $t\in J$, $\un s,\un s'\in\De_n$,
\beglabel{ineqxispxi}
\sum_{i=1}^n \abs{\dxi^t_i(\un s')-\dxi^t_i(\un s)} \le
\ee^{3L/\de} \abs{\ga(a)}\sum_{i=1}^n \abs{s'_i-s_i}.
\elabel

%%%%%%%%%%%%%%%%%%%%%%%%%%%%%%%%%%%%%%%%%%%%%%%%%%%%%
\begin{lem}	\label{lemVarX}
Whenever the function $\eta$ is $1$-Lipschitz on~$\C$ and $\abs{\ga'(\tau)}\le1$
for all $\tau\in J$, the vector field~$X$ defined by \eqref{eqdefD}--\eqref{eqdefX}
satisfies
\beglabel{ineqVarX}
\sum_{i=1}^n \abs{X_i(\tau,\un\ze') - X_i(\tau,\un\ze)} \le
\frac{3}{D(\tau,\un\ze')} \sum_{i=1}^n \abs{\dzeph_i'-\dzeph_i}
\elabel
for any $\tau\in J$ and $\un\ze,\un\ze'\in\gS_\Om^n$.
\end{lem}
%%%%%%%%%%%%%%%%%%%%%%%%%%%%%%%%%%%%%%%%%%%%%%%%%%%%%
\begin{proof}[Proof of Lemma~\ref{lemVarX}]
Let $\tau\in J$ and $\un\ze,\un\ze'\in\gS_\Om^n$. For $i=1,\ldots,n$, we can write
\[
X_i(\tau,\un\ze') - X_i(\tau,\un\ze) = \bigg(
\eta\big(\dzeph_i'\big)-\eta\big(\dzeph_i\big) - 
\big(D(\tau,\un\ze')-D(\tau,\un\ze)\big) \frac{\eta\big(\dzeph_i\big)}{D(\tau,\un\ze)}
\bigg) \frac{\ga'(\tau)}{D(\tau,\un\ze')},
\]
with $D(\tau,\un\ze')-D(\tau,\un\ze) = \sum_{j=0}^n \big( 
\eta\big(\dzeph_j'\big)-\eta\big(\dzeph_j\big)
\big)$,
using the notations $\dzeph_0 = \ga(\tau) - S_n(\un\ze)$,
$\dzeph_0' = \ga(\tau) - S_n(\un\ze')$.
Since $\eta$ is $1$-Lipschitz, we have
$\abs*{\eta\big(\dzeph_j'\big)-\eta\big(\dzeph_j\big)} \le \abs*{\dzeph_j'-\dzeph_j}$ for $j=0,\ldots,n$
and $\abs*{\dzeph_0'-\dzeph_0} \le \sum_{j=1}^n \abs*{\dzeph_j'-\dzeph_j}$,
whence
$\abs*{D(\tau,\un\ze')-D(\tau,\un\ze)} \le \sum_{j=0}^n \abs*{\dzeph_j'-\dzeph_j}
\le 2 \sum_{j=1}^n \abs*{\dzeph_j'-\dzeph_j}$.
The result follows because 
$\sum_{i=1}^n \eta\big(\dzeph_i\big) \le D(\tau,\un\ze)$.
\end{proof}
%%%%%%%%%%%%%%%%%%%%%%%%%%%%%%%%%%%%%%%%%%%%%%%%%%%%%

\begin{proof}[Proof of inequality~\eqref{ineqxispxi}]
Let us fix $\un s,\un s' \in \De_n$ and denote by $\De(t)$ the \lhs\
of~\eqref{ineqxispxi}, \ie
\[
\De(t) = \sum_{i=1}^n \abs{\De_i(t)}, \qquad
\De_i(t) \defeq {\dxi^t_i(\un s')-\dxi^t_i(\un s)}.
\]
For every $t\in J$, we have
\[
\De_i(t) = \De_i(a)
+ \int_a^t \Big( X_i\big(\tau,\un\xi^\tau(\un s')\big)
            -X_i\big(\tau,\un\xi^\tau(\un s)\big) \Big) \,\dd\tau,
\qquad i=1,\ldots, n.
\]
By Lemma~\eqref{lemVarX}, we get
\[
\abs{\De(t)-\De(a)} \le \sum_{i=1}^n \abs{\De_i(t)-\De_i(a)} \le
\int_a^t \frac{3}{D\big(\tau,\un\xi^\tau(\un s')\big)} \De(\tau) \,\dd\tau.
\]
We have seen that $D\big(\tau,\un\xi^\tau(\un s')\big)$ stays $\ge\de$ (this was
point~\textbf{(b)}), thus
$\abs{\De(t)-\De(a)} \le \frac{3}{\de} \int_a^t \De(\tau) \,\dd\tau$ for all
$t\in J$.
Gronwall's lemma yields
\[
\abs{\De(t)} \le \De(a) \, \ee^{3(t-a)/\de}, \qquad t\in J,
\]
and, in view of the initial conditions $\De_i(a) = (s'_i-s_i) \ga(a)$, 
\eqref{ineqxispxi} is proved.
\end{proof}
\medskip

%%%%%%%%%%%%%%%%%%%%%%%%%%%%%%%%%%%%%%%%%%%%%%%%%%%%%

\noindent \textbf{(f)}
Let us fix $t\in J$. For any $\un s\in \De_n$ at which
$(\dxi_1^t,\ldots,\dxi_n^t)$ is differentiable, 
because of~\eqref{ineqxispxi},
the entries of the matrix
$\gJ \defeq \Big[ \frac{\pa\dxi^t_i}{\pa s_j}(\un s)\Big]_{1\le i,j\le n}$
satisfy
\[
\sum_{i=1}^n \abs{\frac{\pa\dxi^t_i}{\pa s_j}(\un s)}
\le \ee^{3L/\de} \abs{\ga(a)},
\qquad j=1,\ldots,n.
\]
We conclude by observing that
\[
\abs{\Det(\gJ)} \le 
\Big(\sum_{i=1}^n \abs{\gJ_{i,1}} \Big) \cdots
\Big(\sum_{i=1}^n \abs{\gJ_{i,n}} \Big) 
\le \big(\ee^{3L/\de} \abs{\ga(a)}\big)^n
\]
(because the \lhs\ is bounded by the sum of the products
$\abs{\gJ_{\sig(1),1}\cdots\gJ_{\sig(n),n}}$
over all bijective maps $\sig \col [1,n] \to [1,n]$, while the middle expression is
equal to the sum of the same products over \emph{all} maps $\sig \col [1,n] \to [1,n]$).
\end{proof}

%%%%%%%%%%%%%%%%%%%%%%%%%%%%%%%%%%%%%%%%%%%%%%%%%%%%%
%%%%%%%%%%%%%%%%%%%%%%%%%%%%%%%%%%%%%%%%%%%%%%%%%%%%%

\begin{proof}[Proof of Theorem~\ref{thmboundconvga}']
Let $\de,L>0$ with $\de < \rho(\Om)/2$ and $\ze \in \cK_{\de,L}(\Om)$. 
We want to prove
\[
\abs{ 1*\hat\ph_1*\cdots*\hat\ph_n(\ze)}
\le \frac{1}{n!} {\big(\rho(\Om)\,\ee^{3L/\de}\big)^n}
\max_{\cK_{\de',L}(\Om)}\abs{\hat\ph_1} \cdots \max_{\cK_{\de',L}(\Om)}\abs{\hat\ph_n}
\]
for any $n\ge1$ and $\hat\ph_1,\ldots,\hat\ph_n \in \hat\gR_\Om$.

We may assume $\ze\not\in \cL_{0_\Om}(\D_{\rho(\Om)})$ (since the behaviour of
convolution products on the principal sheet is already settled
by~\eqref{ineqppalsheet}
and $\ze\in \cL_{0_\Om}(\D_{\rho(\Om)})$ would imply 
$\frac{\abs{\ze}}{n+1} < \rho(\Om)\,\ee^{3L/\de}$).
We can then choose a representative of~$\ze$ in~$\gP_\Om$ which is a
$C^1$ path, the initial part of which is a line segment ending in
$\D_{\rho(\Om)/2}\setminus\D_\de$;
since we prefer to parametrize our paths by arc-length, we take $\ti\ga \col
[\ti a,b] \to \C$ with
$\ti\ga'(t)\equiv1$ and $\leng(\ti\ga) = b-\ti a\le L$, 
and $a\in(\ti a,b)$ such that
\begin{itemize}
\item
$\ti\ga(a) \in \D_{\rho(\Om)/2}$, 
\item
$\ti\ga(t) = \frac{t-\ti a}{a-\ti a}\ti\ga(a)$ for all $t\in [\ti a,a]$,
\item
$\dist\big( \ti\ga(t), \Om \big) \ge\de$ for all $t\in [a,b]$.
\end{itemize}
Now the restriction $\ga$ of~$\ti\ga$ to $[a,b]$ satisfies all the assumptions
of Proposition~\ref{propesimisot},
while formula~\eqref{eqcontgaconvolPsi} of Proposition~\ref{propcontgaconvolPsi}
for $t=b$ can be interpreted as
\beglabel{eqvalueconvolze}
1 * \hat\ph_1 * \cdots * \hat\ph_n(\ze)
= \int_{\De_n}
\hat\ph_1\big(\xi^b_1(\un s)\big) \cdots \hat\ph_n\big(\xi^b_n(\un s)\big)
\Det\bigg[ \frac{\pa\dxi^b_i}{\pa s_j}(\un s)\bigg]_{1\le i,j\le n}
\, \dd s_1\cdots \dd s_n.
\elabel
The conclusion follows immediately, since the Lebesgue measure of~$\De_n$ is $1/n!$.
\end{proof}

%%%%%%%%%%%%%%%%%%%%%%%%%%%%%%%%%%%%%%%%%%%%%%%%%%%%%

We can now prove the main result which was announced in Section~\ref{secmainthm}.

\begin{proof}[Proof of Theorem~\ref{thmboundconvga}]
Let $\de,L>0$ with $\de < \rho(\Om)$, $n\ge1$ and $\hat\ph_1,\ldots,\hat\ph_n
\in \hat\gR_\Om$.
Let $\ze \in \cK_{\de,L}(\Om)$. We must prove
\[
\abs{ \hat\ph_1*\cdots*\hat\ph_n(\ze) }
\le \frac{2}{\de} \cdot \frac{C^n}{n!} \cdot 
\max_{\cK_{\de',L'}(\Om)}\abs{\hat\ph_1} \cdots \max_{\cK_{\de',L'}(\Om)}\abs{\hat\ph_n}. 
\]
One can check that any $\ze' \in \cL_\ze(\D_{\de/2}) = \{\, \ze+w \mid \abs{w} <
\de/2 \,\}$ satisfies
\beglabel{eqdovezep}
\ze' \in \cK_{\de/2,L'}(\Om),
\qquad \text{where $L' \defeq L +\de/2$.}
\elabel
Indeed, $\ze$ is the endpoint of a path~$\ga$ starting from~$0_\Om$, of length
$\le L$, which has $R_\Om\big(\ga(t)\big) \ge \de$. 
In particular $R_\Om(\ze) \ge \de$ thus the path $t\in[0,1] \mapsto
\sig(t)\defeq \ze+t\big(\dzeph'-\dze\big)$ is well-defined.
Either $\ze$ does not lie in the principal sheet of~$\gS_\Om$, then
$\dist(\dze,\Om)\ge\de$ implies $\dist\big(\sig(t),\Om\big)\ge\de/2$ and,
by concatenating~$\ga$ and~$\sig$, we see that \eqref{eqdovezep} holds;
or~$\ze$ is in the principal sheet and then we can choose $\ga$ contained in the
principal sheet and we have at least $\dist\big(\sig(t),\Om\setminus\{0\}\big)\ge\de/2$;
if moreover $\dze\in\D_{\rho(\Om)}$ then also $\sig$ is contained in the principal
sheet, with $R_\Om\big(\sig(t)\big)\ge \de/2$,
whereas if $\dze\not\in\D_{\rho(\Om)}$ then 
$\dist\big(\sig(t),\{0\}\big)\ge \rho(\Om)-\de/2 \ge \de/2$,
hence again $R_\Om\big(\sig(t)\big)\ge \de/2$, thus~\eqref{eqdovezep} holds in all cases.

%%%%%%%%%%%%%%%%%%%%%%%%%%%%%%%%%%%%%%%%%%%%%%%%%%%%%

Thus, by Theorem~\ref{thmboundconvga}',
\[
\max_{\cL_\ze(\D_{\de/2})} \abs{ 1*\hat\ph_1*\cdots*\hat\ph_n }
\le \frac{C^n}{n!}
\max_{\cK_{\de',L'}(\Om)}\abs{\hat\ph_1} \cdots \max_{\cK_{\de',L'}(\Om)}\abs{\hat\ph_n}
\]
with $\de' \defeq \demi \rho(\Om) \, \ee^{-4L'/\de}$ and 
$C \defeq \rho(\Om)\,\ee^{6L'/\de}$, which are precisely the values indicated
in~\eqref{eqdefconstants}. 
The conclusion follows from the Cauchy inequalities.
\end{proof}

%%%%%%%%%%%%%%%%%%%%%%%%%%%%%%%%%%%%%%%%%%%%%%%%%%%%%

\begin{rem}	\label{remmistCNP}
As far as we understand, there is a mistake in \cite{CNP}, in the final argument
given to bound a determinant analogous to our formula~\eqref{ineqDet}:
roughly speaking, these authors produce a deformation of the standard
$n$-simplex through the flow of an autonomous vector field in~$\C^n$ (the
definition of which is not clear to us) and then use the linear differential
equation satisfied by the Jacobian determinant of the flow; however, they
overlook the fact that, since their vector field is not holomorphic, the
Jacobian determinant which can be controlled this way is the real one,
corresponding to the identification $\C^n \simeq \R^{2n}$, whereas the determinant
which appears when computing the integral and that one needs to bound is a
complex linear combination of the $n\times n$ minors of the $2n \times 2n$ real
Jacobian matrix.
\end{rem}

%%%%%%%%%%%%%%%%%%%%%%%%%%%%%%%%%%%%%%%%%%%%%%%%%%%%%
%%%%%%%%%%%%%%%%%%%%%%%%%%%%%%%%%%%%%%%%%%%%%%%%%%%%%

\appendix

\section{Appendix: a class of rectifiable currents 
and their Lipschitz push-forwards}	\label{appNorCur} 

%%%%%%%%%%%%%%%%%%%%%%%%%%%%%%%%%%%%%%%%%%%%%%%%%%%%%

In this appendix, we single out a few facts from Geometric Measure Theory which
are useful in the proof of our main result.
Among the standard references on the subject one can quote
\cite{Federer}, \cite{Simon}, \cite{AmbrosioK}, \cite{Morgan}.

For a differentiable manifold~$M$ and an integer $m\ge0$, we denote
by~$\gE_m(M)$ the space of all $m$-dimensional currents with compact support,
viewed as linear functionals on the space of all $C^\infty$ differential $m$-forms (with
complex-valued coefficients) which are continuous for the usual family of
seminorms (defined by considering the partial derivatives of the coefficients of
forms in compact subsets of charts).
In fact, by taking real and imaginary parts, the situation is reduced to that of
real-valued forms and real-valued currents.
For us, $M=\R^N$ or $M=\gS_\Om^n$, but in the latter case, as far as currents
are concerned, the local biholomorphism $\pi_\Om^{\otimes n}$ makes the
difference between $\gS_\Om^n$ and~$\C^n$ immaterial, and the complex structure
plays no role, so that one loses nothing when replacing~$M$ with~$\R^{2n}$.

%%%%%%%%%%%%%%%%%%%%%%%%%%%%%%%%%%%%%%%%%%%%%%%%%%%%%
\subsection*{Integration currents associated with oriented compact rectifiable sets}
%%%%%%%%%%%%%%%%%%%%%%%%%%%%%%%%%%%%%%%%%%%%%%%%%%%%%

Let $m, N\in\N^*$. We denote by~$\gH^m$ the $m$-dimensional Hausdorff measure
in~$\R^N$.
A basic example of $m$-dimensional current in~$\R^N$ is obtained as follows:

%%%%%%%%%%%%%%%%%%%%%%%%%%%%%%%%%%%%%%%%%%%%%%%%%%%%%
\begin{Def}
Let~$S$ be an {oriented compact $m$-dimensional rectifiable subset of~$\R^N$} 
(\ie $S$ is compact, $\gH^m$-almost all of~$S$
is contained in the union of the images of countably many Lipschitz maps
from~$\R^m$ to~$\R^N$
and we are given a measurable orientation of the approximate tangent $m$-planes%
\footnote{
Recall that, at $\gH^m$-almost every point of~$S$, the cone of approximate
tangent vectors is an $m$-plane \cite[3.2.19]{Federer}, \cite[3.12]{Morgan}).
}
to~$S$)
and, for $\gH^m$-a.e.\ $x\in S$, let $\tau(x)$ be a unit $m$-vector orienting the
tangent $m$-plane at~$x$;
then the formula
\beglabel{eqdefcrocS}
[S] \col \al \;\text{$m$-form on $\R^N$} \mapsto 
\int_S \croc{ \tau(x),\al(x) } \,\dd\gH^m(x)
\elabel
defines a current $[S] \in \gE_m(\R^N)$, the support of which is~$S$.
\end{Def}
%%%%%%%%%%%%%%%%%%%%%%%%%%%%%%%%%%%%%%%%%%%%%%%%%%%%%

This example belongs to the class of \emph{integer rectifiable currents}, for
which the \rhs\ of~\eqref{eqdefcrocS} more generally assumes the form
\[\int_S \croc{ \tau(x),\al(x) } \mu(x) \,\dd\gH^m(x),\]
where~$\mu$ is a \emph{multiplicity function}, \ie an $\gH^m$-integrable
function $\mu \col S \to \N^*$.

One must keep in mind that a rectifiable current is determined by a triple
$(S,\tau,\mu)$ where the orienting $m$-vector~$\tau$ is tangent to the
support~$S$ (at $\gH^m$-almost every point);
this is of fundamental importance in what follows (taking an $m$-vector
field~$\tau$ which is not tangent to~$S$ almost everywhere would lead to very
different behaviours when applying the boundary operator).
In this appendix we shall content ourselves with the case $\mu\equiv1$.

An elementary example is $[\De_N] \in \gE_N(\R^N)$, with the standard
$N$-dimensional simplex $\De_N \subset \R^N$ of Notation~\ref{notesimplexn}
oriented by 
$\tau = \frac{\pa\;\;\;}{\pa x_1} \wedge \cdots \wedge \frac{\pa\;\;\;}{\pa x_N}$.
%

%%%%%%%%%%%%%%%%%%%%%%%%%%%%%%%%%%%%%%%%%%%%%%%%%%%%%
\subsection*{Push-forward by smooth and Lipschitz maps}
%%%%%%%%%%%%%%%%%%%%%%%%%%%%%%%%%%%%%%%%%%%%%%%%%%%%%

The push-forward of a current $T\in\gE_m(\R^N)$ by a smooth map $\Phi \col \R^N
\to \R^{N'}$ is classically defined by dualizing the pullback of differential forms:
\[
\Phi_\# T(\be) \defeq T(\Phi^\#\be),
\qquad \text{$\be$ any $m$-form on $\R^{N'}$,}
\]
which yields $\Phi_\# T \in \gE_m(\R^{N'})$.

For an integration current $T=[S]$ as in~\eqref{eqdefcrocS}, we observe that
the smoothness of~$\al$ is not necessary for the definition of $[S](\al)$ to make sense:
it suffices that $\al$ be defined $\gH^m$-almost everywhere on~$S$, bounded and
$\gH^m$-measurable.
Therefore, \emph{in the top-dimensional case $m=N$, we can associate with the current
$[S]\in\gE_N(\R^N)$ a push-forward $\phi_\#[S] \in \gE_N(\R^{N'})$ by any
Lipschitz map $\phi\col S\to \R^{N'}$, by means of the formula
\beglabel{eqdefphidiesisLip}
\phi_\#[S](\be) \defeq [S](\phi^\#\be),
\qquad \text{$\be$ any $N$-form on $\R^{N'}$.}
\elabel
} 
Indeed, Rademacher's theorem ensures that $\phi$ is differentiable $\gH^N$-almost
everywhere ($\gH^N$ is the Lebesgue measure), with bounded partial
derivatives, hence the pullback form $\phi^\#\be$ is defined almost everywhere
as
\begin{multline*}
\be = \sum_I g_I \, \dd y^{I_1} \wedge \cdots \wedge \dd y^{I_N}
\quad \Longrightarrow \\[1ex]%\quad
\phi^\#\be = \sum_I (g_I\circ\phi) \, 
\dd\phi^{I_1} \wedge \cdots \wedge \dd\phi^{I_N}
= \sum_I (g_I\circ\phi) \Det\bigg[ 
\frac{\pa\phi^{I_i}}{\pa x^j\,} \bigg]_{1\le i,j\le N}
\dd x^1 \wedge \cdots \wedge x^N,
\end{multline*}
where the sums are over all $I = \{ 1 \le I_1 < \cdots < I_N \le N' \}$,
the coordinates in~$\R^{N'}$ are denoted by $(y^1,\ldots,y^{N'})$
and those in~$\R^{N}$ by $(x^1,\ldots,x^{N})$.
The pullback form $\al = \phi^\#\be$ has its coefficients in $L^\infty(\R^N)$,
hence we can define $\phi_\#[S](\be) = [S](\al)$ by~\eqref{eqdefcrocS}.

%%%%%%%%%%%%%%%%%%%%%%%%%%%%%%%%%%%%%%%%%%%%%%%%%%%%%

Having defined $\phi_\#[S]\in\gE_N(\R^{N'})$ by formula~\eqref{eqdefphidiesisLip},
it is worth noticing that $\phi_\#[S]$ can also be obtained by a regularization
process:

%%%%%%%%%%%%%%%%%%%%%%%%%%%%%%%%%%%%%%%%%%%%%%%%%%%%%
\begin{lem}
Let $S$ be an oriented compact $N$-dimensional rectifiable subset of~$\R^N$ and
let $\phi \col S \to \R^{N'}$ be a Lipschitz map.
Consider smooth Lipschitz maps $\Phi_\ell \col \R^N \to \R^{N'}$, $\ell\in\N$,
which have uniformly bounded Lipschitz constants
and converge uniformly to~$\phi$ on~$S$ as $\ell\to\infty$.
Then 
\beglabel{eqconvapproxbeta}
(\Phi_\ell)_\#[S](\be) \xrightarrow[\ell\to\infty]{} \phi_\#[S](\be),
\qquad \text{$\be$ any $N$-form on $\R^{N'}$.}
\elabel 
\end{lem}
%%%%%%%%%%%%%%%%%%%%%%%%%%%%%%%%%%%%%%%%%%%%%%%%%%%%%

The proof relies on equicontinuity estimates derived from Reshetnyak's theorem%
\footnote{See \cite{Evans}, \S~8.2.4, Lemma on the weak continuity of determinants.}
which guarantees that in this situation, not only do we have the weak-$*$
convergence in $L^\infty(\R^N)$ for the partial derivatives 
$\frac{\pa\Phi_\ell^{I_k}}{\pa x_j\,} 
\underset{^*}{\rightharpoonup} 
\frac{\pa\phi^{I_k}}{\pa x_j\,}$,
but also for the minors of the Jacobian matrix:
$\Det\big[ \frac{\pa\phi_\ell^{I_i}}{\pa x^j\,} \big]
\underset{^*}{\rightharpoonup} 
\Det\big[ \frac{\pa\phi^{I_i}}{\pa x^j\,} \big]$,
whence $\Phi_\ell^\#\be
\underset{^*}{\rightharpoonup} 
\phi^\#\be$ componentwise in $L^\infty(\R^N)$ and \eqref{eqconvapproxbeta} follows.

%%%%%%%%%%%%%%%%%%%%%%%%%%%%%%%%%%%%%%%%%%%%%%%%%%%%%

Another case of interest is $T = [A(\De)] \in \gE_m(\R^N)$ with $m \le N$, $\De$
an oriented compact $m$-dimensional rectifiable subset of~$\R^m$ and $A \col
\R^m \to \R^N$ an injective affine map
(the unit $m$-vector field orienting $A(\De)$ is chosen to be a positive
multiple of the image of the unit $m$-vector field orienting~$\De$ by the
$m$-linear extension of the linear part of~$A$ to $\La_m\R^m$).
We have $[A(\De)] = A_\#[\De]$, thus the natural definition of the push-forward
of $[A(\De)]$ by a Lipschitz map
$\phi\col A(\De) \to \R^{N'}$ is clearly
\beglabel{eqdefphifiesiszeroDe}
\phi_\#[A(\De)] \defeq (\phi_A)_\#[\De] \in \gE_{N-1}(\R^{N'}),
\qquad \text{with $\phi_A \defeq \phi\circ A \col \De \to \R^{N'}$.}
\elabel
Indeed, one easily checks that when $\phi$ is the restriction to
$A(\De)$ of a smooth map $\Phi \col \R^N \to \R^{N'}$, the above-defined
push-forward $\phi_\#[A(\De)]$ coincides with the classical
push-forward $\Phi_\#[A(\De)]$.
Moreover, also in this case is the regularization process possible:
for any sequence of smooth Lipschitz maps $\Phi_\ell \col \R^N \to \R^{N'}$, $\ell\in\N$,
which have uniformly bounded Lipschitz constants
and converge uniformly to~$\phi$ on $A(\De)$ as $\ell\to\infty$,
we have
\beglabel{eqconvapproxbetabis}
(\Phi_\ell)_\#[A(\De)](\be) \xrightarrow[\ell\to\infty]{} \phi_\#[A(\De)](\be),
\qquad \text{$\be$ any $N$-form on $\R^{N'}$}
\elabel 
(simply because the \lhs\ is $(\Phi_\ell\circ A)_\#[\De](\be)$ and we can
apply~\eqref{eqconvapproxbeta} to the sequence $\Phi_\ell\circ A$ uniformly
converging to $\phi\circ A$ on~$\De$).

%%%%%%%%%%%%%%%%%%%%%%%%%%%%%%%%%%%%%%%%%%%%%%%%%%%%%
\subsection*{The boundary operator and Stokes's theorem}
%%%%%%%%%%%%%%%%%%%%%%%%%%%%%%%%%%%%%%%%%%%%%%%%%%%%%

The boundary operator is defined by duality on all currents $T\in\gE_m(\R^N)$:
\beglabel{eqdefboundaryop}
\pa T(\al) \defeq T(\dd\al), \qquad \text{$\al$ $m$-form on $\R^N$.}
\elabel
The boundary of an integer rectifiable current~$T$ is not necessarily an integer rectifiable current;
if it happens to be, then $T$ is called an \emph{integral current}.
An example is provided by oriented smooth submanifolds~$M$ with boundary; Stokes's
theorem then relates the action of the boundary operator~$\pa$ on the corresponding
integration currents with the action of the boundary operator~$\pa$ of homology:
$\pa [M] = [\pa M] \in \gE_{m-1}(\R^N)$.

Another example is provided by the standard simplex $\De_N \subset \R^N$ of Notation~\ref{notesimplexn};
recall that the orienting unit $n$-vector field is
$\tau \defeq \frac{\pa\;\;\;}{\pa x_1} \wedge \cdots \wedge \frac{\pa\;\;\;}{\pa x_N}$.
Stokes's Theorem yields
\[
\pa[\De_N] = [\Ga_0] + \cdots + [\Ga_N] \in \gE_{N-1}(\R^N),
\]
where 
\[
\Ga_j = \left| \begin{aligned}
&\De_N \cap \{ x_1+\cdots + x_N = 1 \} & &\text{if $j=0$,} \\[1ex]
&\De_N \cap \{ x_j = 0 \} & &\text{if $1\le j \le n,$} \\[1ex]
\end{aligned} \right.
\]
with orienting $(N-1)$-vectors $\tau_j$ defined by $\nu_j \wedge \tau_j =
\tau$, where $\nu_j$ is the outward-pointing unit normal vector field for the
piece~$\Ga_j$ of $\pa \De_N$;
with the notation $e_j = \frac{\pa\;\;\;}{\pa x_j}$, the result is
$\tau_0 = \frac{1}{\sqrt N} (e_2-e_1) \wedge (e_3-e_1) \wedge \cdots \wedge (e_N-e_1)$
(because $\nu_0 = (e_1+\cdots+e_N)/\sqrt N$) and
$\tau_j = (-1)^j e_1 \wedge \cdots \wedge \htab{e_j} \wedge \cdots \wedge e_N$
for $j\ge1$ (because $\nu_j = -e_j$).

Observe that one can write
\beglabel{eqpaDenaffine}
\pa[\De_N] = [A_0(\De_{N-1})] - [A_1(\De_{N-1})] + \cdots + (-1)^N [A_N(\De_{N-1})]
\elabel
with an injective affine map $A_j \col \R^{N-1} \to \R^N$ for each $j=0,\ldots,N$
(taking $A_0(x_1,\ldots,x_{N-1}) = (1-x_1-\cdots-x_{N-1},x_1,\ldots,x_{N-1})$ and
$A_j(x_1,\ldots,x_{N-1}) = (x_1,\ldots,x_{j-1},0, x_j,\ldots,x_{N-1})$ for $j\ge1$).

%%%%%%%%%%%%%%%%%%%%%%%%%%%%%%%%%%%%%%%%%%%%%%%%%%%%%
\subsection*{The commutation formula $\phi_\# \pa [P] = \pa \phi_\# [P]$}
%%%%%%%%%%%%%%%%%%%%%%%%%%%%%%%%%%%%%%%%%%%%%%%%%%%%%

For any $T\in\gE_m(\R^N)$ and any smooth map $\Phi \col \R^N \to \R^{N'}$, the
formula
\beglabel{eqsmoothcomm}
\Phi_\# \pa T = \pa \Phi_\# T \in \gE_{m-1}(\R^{N'})
\elabel
is a simple consequence of the identity $\dd\circ \Phi^\# = \Phi^\#\circ\dd$ on
differential forms.
We can also try to deal with a Lipschitz map~$\phi$ when restricting ourselves to integral
currents.
The following is used in the proof of the main result of this article:

%%%%%%%%%%%%%%%%%%%%%%%%%%%%%%%%%%%%%%%%%%%%%%%%%%%%%

\begin{lem}	\label{lempushcomm}
Let $N\ge1$ and let $\phi \col \De_N \to \R^{N'}$ be Lipschitz;
define $\phi_\#[\De_N]$ by means of~\eqref{eqdefphidiesisLip}
and $\phi_\#\pa[\De_N]$ by means of~\eqref{eqpaDenaffine} and~\eqref{eqdefphifiesiszeroDe}.
Then
\[
\pa\phi_\#[\De_N] = \phi_\#\pa[\De_N].
\]
\end{lem}

%%%%%%%%%%%%%%%%%%%%%%%%%%%%%%%%%%%%%%%%%%%%%%%%%%%%%

In fact, it is with $P = [0,1] \times \De_n$ instead of $\De_N$ that this
commutation formula is used in Section~\ref{secdefor}; 
moreover, the target space is~$\gS_\Om^n$ instead of~$\R^{N'}$ but, as mentioned
above, this makes no difference (just take $N'=2n$).
We leave it to the reader to adapt the proof.

%%%%%%%%%%%%%%%%%%%%%%%%%%%%%%%%%%%%%%%%%%%%%%%%%%%%%

\begin{proof}[Proof of Lemma~\ref{lempushcomm}]
We shall use the notation $T = [\De_N] \in \gE_N(\R^N)$.
Let $\be$ be a smooth $(N-1)$-form on $\R^{N'}$ and let
$(\Phi_\ell)_{\ell\in\N}$ be any sequence of smooth maps from $\R^N$
to~$\R^{N'}$ with uniformly bounded Lipschitz constants which converges
uniformly to~$\phi$ on~$\De_N$ as $\ell\to\infty$.
Then the sequence
\[
\pa (\Phi_\ell)_\# T (\be) = (\Phi_\ell)_\# T (\dd\be)
\xrightarrow[\ell\to\infty]{}
\phi_\# T (\dd\be) = \pa\phi_\# T (\be)
\]
by~\eqref{eqdefboundaryop} and~\eqref{eqconvapproxbeta}.
But, by~\eqref{eqsmoothcomm}, this sequence coincides with
\[
(\Phi_\ell)_\# \pa T (\be) = \sum_{j=0}^N (-1)^j (\Phi_\ell)_\#[A_j(\De_{N-1})](\be)
\xrightarrow[\ell\to\infty]{}
\sum_{j=0}^N (-1)^j \phi_\# [A_j(\De_{N-1})](\be) =  \phi_\# \pa T (\be)
\]
by~\eqref{eqpaDenaffine} and~\eqref{eqconvapproxbetabis}.
\end{proof}

%%CCCCCCCCCCCCCCCCCCCCCCCCCCCCCCCCCCCCCCC%%

\vspace{.45cm}

%%%%%%%%%%%%%%%%%%%%%%%%%%%%%%%%%%%%
%%%%%%%%%%%%%%%%%%%%%%%%%%%%%%%%%%%%

\noindent {\em Acknowledgements.}
%
%%%%%%%%%%%%%%%%%%%%%%%%%%%%%%%%%%%%%%%%%%%%%%%%%%%%%
%
We thank M.~Abate, E.~Bedford and J.~\'Ecalle for useful discussions.
We express special thanks to L.~Ambrosio for his help on the theory of currents.
The author is also indebted to the anonymous referees for their help
in improving this article.

The author acknowledges the support of the Centro di Ricerca
Matematica Ennio de Giorgi.
The research leading to these results was partially supported by the
European Comunity's Seventh Framework Program (FP7/2007--2013) under
Grant Agreement n.~236346
and by the French National Research Agency under the reference
ANR-12-BS01-0017.

%%%%%%%%%%%%%%%%%%%%%%%%%%%%%%%%%%%%%%%%%%%%%%%%%%%%%%%
%%%%%%%%%%%%%%%%%%%%%%%%%%%%%%%%%%%%%%%%%%%%%%%%%%%%%%%

%\vspace{.5cm}

% \newpage

%%%%%%%%%%%%%%%%%%%%%%%%%%%%%%%%%%%%%%%%%%%%%%%%%%%%%%%%
%%%%%%%%%%%%%%%%%%%%%%%%%%%%%%%%%%%%%%%%%%%%%%%%%%%%%%%%

\vspace{.55cm}

\noindent
David Sauzin\\[1ex]
CNRS UMI 3483 - Laboratorio Fibonacci \\
Centro di Ricerca Matematica Ennio De Giorgi, \\
Scuola Normale Superiore di Pisa \\
Piazza dei Cavalieri 3, 56126 Pisa, Italy\\
email:\,{\tt{david.sauzin@sns.it}}

\end{document}